\documentclass[12pt]{amsart}
\usepackage{amsmath,amssymb,amsthm,hyperref}
\usepackage{graphicx}
\usepackage{tikz}  
 \newtheorem{theorem}{Theorem}
 \newtheorem{corollary}[theorem]{Corollary}
 
 \newtheorem{lemma}[theorem]{Lemma}
 \newtheorem{proposition}[theorem]{Proposition}
 
 {\theoremstyle{definition}
 
 \newtheorem{remark}{Remark}

 }

\newcommand{\N}{\ensuremath{\mathbb N}} %natural numbers
\newcommand{\R}{\ensuremath{\mathbb R}} %real numbers
\newcommand{\norm}[1]{\left\lVert#1\right\rVert}

\newcommand{\E}{\ensuremath{\mathbb E}}

\title[Spectral CLT]{Spectral central limit theorem for additive functionals of isotropic and stationary Gaussian fields}

\author{Leonardo Maini}
\address{Leonardo Maini, Universit\'e du Luxembourg, 
Unit\'e de Recherche en Math\'ematiques,
Maison du Nombre,
6 avenue de la Fonte,
L-4364 Esch-sur-Alzette,
Grand Duch\'e du Luxembourg}
\email{leonardo.maini@uni.lu} 
\thanks{}
\author{Ivan Nourdin}
\address{Ivan Nourdin, Universit\'e du Luxembourg, 
Unit\'e de Recherche en Math\'ematiques,
Maison du Nombre,
6 avenue de la Fonte,
L-4364 Esch-sur-Alzette,
Grand Duch\'e du Luxembourg}
\email{ivan.nourdin@uni.lu} 
\thanks{}

\date{\today}
\begin{document}
\maketitle

\medskip

\begin{abstract}
Let $B=(B_x)_{x\in\R^d}$ be a collection of $N(0,1)$ random variables forming a real-valued continuous stationary Gaussian field on $\R^d$,
and 
set $C(x-y)=\E[B_xB_y]$. Let $\varphi:\R\to\R$ be such that $\E\left[\varphi(N)^2\right]<\infty$ with $N\sim N(0,1)$, let $R$ be the Hermite rank of $\varphi$, and consider
$Y_t = \int_{tD} \varphi(B_x)dx$, $t>0$ with $D\subset \R^d$ compact.

Since the pioneering works from the 80s by Breuer, Dobrushin, Major, Rosenblatt, Taqqu and others, central and noncentral limit theorems for $Y_t$ 
have been constantly refined, extended
and applied to an increasing number of diverse situations, to such an extent that it has become a field of research in its own right.

The common belief, representing the intuition that specialists in the subject have developed over the last four decades, is
that as $t\to\infty$ the fluctuations of $Y_t$ around its mean are, in general (i.e. except possibly in very special cases), Gaussian when $ B $ has short memory,
and non Gaussian when $B$ has long memory and $R\geq 2$.

We show in this paper that this intuition forged over the last forty years can be wrong, and not only marginally or in critical cases.
We will indeed bring to light a variety of situations where $ Y_t $ admits Gaussian fluctuations in a long memory context.

To achieve this goal, we state and prove a spectral central limit theorem, which extends the conclusion of the celebrated Breuer-Major theorem to situations where $C\not\in L^R(\R^d)$. 
Our main mathematical tools are the Malliavin-Stein method and Fourier analysis techniques.\\
\\
{ {\it Keywords: }  Spectral central limit theorem, stationary Gaussian fields, isotropic Gaussian fields, long memory, short memory, Malliavin-Stein method, Hermite rank, Fourier analysis.\\
\\
%{\it AMS 2010 Classification: }  .
}
\end{abstract}

%%%%%%%%%%%%%%%%%%%%%%%%%
%%%%%%%%%%%%%%%%%%%%%%%%%
%%%%%%%%%%%%%%%%%%%%%%%%%
%%%%%%%%%%%%%%%%%%%%%%%%%
%%%%%%%%%%%%%%%%%%%%%%%%%
%%%%%%%%%%%%%%%%%%%%%%%%%
%%%%%%%%%%%%%%%%%%%%%%%%%
%%%%%%%%%%%%%%%%%%%%%%%%%
%%%%%%%%%%%%%%%%%%%%%%%%%
%%%%%%%%%%%%%%%%%%%%%%%%%
%%%%%%%%%%%%%%%%%%%%%%%%%
%%%%%%%%%%%%%%%%%%%%%%%%%
%%%%%%%%%%%%%%%%%%%%%%%%%
\section{Introduction }\label{introduction}

Fix a dimension $d\geq 2$, and consider a real-valued almost surely continuous Gaussian field $(B_x)_{x\in\R^d}$ defined on $\R^d$.
Assume furthermore that $B$ is {\bf stationary}, that is, there is a function $C:\R^d\to\R$ such that 
\begin{equation}\label{C}
{\rm Cov}(B_x,B_y)=C(x-y),\quad x,y\in\R^d,
\end{equation}
and suppose  that $B_x\sim N(0,1)$ for all $x\in\R^d$ or, equivalently, that $\E[B_x]=0$ and $C(0)=1$.

As a second ingredient, consider a measurable function $\varphi:\mathbb{R}\rightarrow\mathbb{R}$ such that
\begin{equation}\label{varphii}
\E\left[\varphi(N)^2\right]<\infty,\quad \mbox{for $N\sim N(0,1)$}.
\end{equation}

Our object of interest in this paper is
\begin{equation}\label{Yt}
Y_t = \int_{tD} \varphi(B_x)dx,\quad t>0,
\end{equation}
where $D\subset \R^d$ is compact with ${\rm Vol}(D)>0$ and $tD:=\{tx|x\in D\}$.
Well-posedness of (\ref{Yt}) as a random variable in $L^2(\Omega)$ is ensured by Proposition \ref{wellposed} and the almost sure continuity of $(B_x)_{x\in\R^d}$.
We also observe that the continuity of $B$, together with its stationarity, implies\footnote{Indeed, since $|B_0(B_{x+h}-B_x)|\le B_0^2+\frac12B_x^2+\frac12B_{x+h}^2$, by applying the generalized dominated convergence theorem we have
$
C(x+h)-C(x)=\E\left[B_0(B_{x+h}-B_x)\right]\rightarrow0$ as $h\rightarrow0$.} the continuity of its covariance function $C$, a property that will be needed to evoke Bochner's theorem later in (\ref{fourier}). 
\medskip

Studying the asymptotic behavior of functionals of the form (\ref{Yt}) dates back from the eighties, with seminal works by
Breuer and Major \cite{BM}, Dobrushin and Major \cite{DM}, Rosenblatt \cite{Rosenblatt} and Taqqu \cite{Taqqu}. 
Since then, limit theorems for (\ref{Yt}) have been constantly investigated, and represent nowadays a central theme in the  modern probability theory.

We note that many interesting geometric quantities associated with the Gaussian field $(B_x)_{x\in\R^d}$ can be represented as functionals of the form (\ref{Yt}). For instance, the
choice $\varphi=\mathbf{1}_{(-\infty,u]}$ (resp. $\varphi=\mathbf{1}_{[u,\infty)}$), $u\in\R$
corresponds to the volume of the lower (resp. upper) level sets of $(B_x)_{x\in\R^d}$.
\medskip

Since (\ref{varphii}) holds, we can decompose $\varphi$ in Hermite polynomials (see, e.g., \cite[Section 1.4]{bluebook}) as
\begin{equation}\label{hermitedecomp}
\varphi =\E[\varphi(N)]+ \sum_{q=R}^\infty a_q H_q,\quad \mbox{with $R\geq 1$ such that $a_R\neq 0$},
\end{equation}
where $H_q$ denotes the $q$th Hermite polynomial and $a_q=a_q(\varphi)=\frac{1}{q!}\mathbb{E}\left[\varphi(N)H_q(N)\right]\in\R$.
The integer $R\geq 1$ is called the {\bf Hermite rank} of $\varphi$. 
We also define the {\bf{second Hermite rank}} $R'\ge2$ of $\varphi$ as the Hermite rank of $\varphi(x)-\E[\varphi(N)]-a_RH_R(x)$ (if $\varphi(x)=\E[\varphi(N)]+a_RH_R(x)$, we set $R'=\infty$).

\medskip 

 In the present paper, we are more specifically  interested in the asymptotic behavior of 
 \begin{equation}\label{normalized}
 	\frac{Y_t-m_t}{\sigma_t},\quad t\to\infty,
 \end{equation}
 where we have $Y_t\in L^2(\Omega)$ for all $t$, and where we note $m_t=\E[Y_t]$ and $\sigma_t = \sqrt{{\rm Var}(Y_t)}>0$. 
 For simplicity, to ensure that $\sigma_t>0$ for all $t$, we will assume for the rest of the paper the existence of some $k\geq 1$ such that $a_{2k}\neq0$\footnote{This comes from the fact that $\sigma_t^2={\rm Var}(Y_t)=\sum_{q=R}^\infty q!a_q^2\int_{(tD)^2}C^q(x-y)dxdy$.}, or equivalently that $\varphi$ is \textbf{not an odd function}. 
As an illustration of what may happen when $\varphi$ is odd, see (\ref{extrank1}).
 
\medskip
Throughout all the paper, for two functions $f,g:\R_+\to\R_+$ we write 
\begin{equation}\label{asymp}
	f(t)\asymp g(t)
\end{equation} 
to indicate that $f(t)=O(g(t))$ and $g(t)=O(f(t))$ as $t\to\infty$.
\subsection{Previous results}
Given its importance in our paper, we start with the celebrated Breuer-Major theorem, stated here in its continuous form.

\begin{theorem}[Breuer, Major \cite{BM}]\label{BM}
Let $B=(B_x)_{x\in\R^d}$ be a real-valued continuous centered Gaussian field on $\R^d$, assumed to be stationary and to have unit-variance.
Let $\varphi:\R\to\R$ be such that $\E[\varphi(N)^2]<\infty$ with $N\sim N(0,1)$, let $R$ be the Hermite rank of $\varphi$,
consider $Y_t$ defined by (\ref{Yt}), and 
recall the definition  (\ref{C}) of the covariance function $C$.

If $\int_{\R^d}|C(x)|^Rdx<\infty$ , then
$
t^{-\frac{d}2}(Y_t-\E[Y_t])\overset{\rm law}{\to} N(0,\sigma^2)$ where 
\begin{equation}
	\label{varBM}\sigma^2={\rm Vol}(D)\sum_{q=R}^\infty q!a_q^2\int_{\R^d}C(z)^qdz\geq 0.
\end{equation}
In particular, if $\varphi$ is not odd, then $\sigma^2>0$, $\sigma_t^2\asymp t^d$ and 
\[
\frac{Y_t-m_t}{\sigma_t}\overset{\rm law}{\to}N(0,1).
\]
\end{theorem}

To describe the asymptotic behavior in the case where 
$C\not\in L^R(\R^d)$ (that is, when we cannot apply the Breuer-Major Theorem \ref{BM}),
we have to be more precise on the behavior of $C$ at infinity.
In the papers studying limit theorem for (\ref{normalized}) in a {\it general} framework (i.e. not for a {\it particular} model),
it is often (if not always) assumed that 
\begin{equation}\label{classique}
C(x)= |x|^{-\beta} L(|x|),
\end{equation}
with $\beta\in(0,\infty)$ and $L:\R_+\to\R$ \textbf{slowly varying} (that is, satisfying $L(\lambda r)/L(\lambda)\rightarrow 1$ as $\lambda\rightarrow\infty$, for every fixed $r>0$).
The following three different situations (i)-(ii)-(iii) then occur:
\begin{enumerate}
	\item If $\beta>\frac{d}{R}$ then $C\in L^R(\R^d)$ and we say that we are in the {\bf short-memory case}. We deduce from  Breuer-Major Theorem \ref{BM} that $t^{-d/2}(Y_t-\E[Y_t])\overset{\rm law}{\to}  N(0,\sigma^2)$.
	\item If $\beta=\frac{d}{R}$ we are in a {\bf critical case}. Although $C\not\in L^R(\R^d)$, fluctuations of $Y_t$ around its mean are still asymptotically Gaussian (see, e.g., \cite[Theorem 1']{BM}, \cite[Section 5]{NN} and \cite[Theorem 1]{NNT}).
	\item If $\beta<\frac{d}{R}$ then $C\not\in L^R(\R^d)$ and we say that we are in the {\bf long-memory case}.
	A theorem of Dobrushin and Major \cite{DM} asserts that $t^{-(d-\frac{R\beta}{2})}L^{-R/2}(t)(Y_t-\E[Y_t])\overset{\rm law}{\to} Z$ where, up to a multiplicative constant, $Z$ is the Hermite distribution of order $R$ 
	and self-similarity index $H\in(0,1)$ (with $H$ depending only on $\beta$).
	Since we do not use it in the sequel, we do not give its precise definition here. (The interested reader can
	e.g. consult \cite{tudor-book}.) Let us only stress here that the Hermite distribution of order $R$ belongs to the $R$th Wiener chaos, and so is {\bf not Gaussian} as soon as $R\geq 2$.
\end{enumerate}

\subsection{Motivating examples}
The intuition we can naturally develop from the previous (i)-(ii)-(iii) (and that represents the common intuition forged by the papers written on the subject over the last forty years) is that $Y_t$ defined by $(\ref{Yt})$ displays {\it Gaussian} (resp. {\it non-Gaussian}) fluctuations 
when the Gaussian field $B$ has {\it short} (resp. {\it long}) memory, and this whatever the function $\varphi$ with 
Hermite rank $R\geq 2$. (The case $R=1$ is apart.\footnote{The case where $R=1$ is apart since, whatever the memory, we generally obtain Gaussian fluctuations. But this is for different reasons that we understand better in the functional version: in the short memory case the limit is a standard Brownian motion, while in the long memory case the limit is a fractional Brownian motion.})

As anticipated, we will show in this paper that this intuition can be wrong, and not only marginally or in critical cases.
We will indeed state a central limit theorem (Theorem \ref{thm:main} below) whose conclusion is valid provided that
a certain spectral condition is satisfied. As we will see, this may lead to Gaussian fluctuations in a long memory context, in total contrast with Dobrushin and Major \cite{DM} (see (iii) above).

Recently, {\bf Berry's Random Wave Model} (BRWM) has attracted a lot of attention.
It is defined as follows. Choose the dimension $d=2$ and consider the centered continuous Gaussian field $B$ on $\R^2$
with covariance $\E[B_xB_y]=C(x-y)=J_0(|x-y|)$, with $J_0$ the Bessel function of the first kind of order 0 (see Section \ref{besselstuff} for its definition and some properties);
in particular, we have
\begin{equation}\label{pasclassique}
C(x)= \sqrt{\frac{2}{\pi}}\,|x|^{-\frac12}\,\cos\left(|x|-\frac{\pi}4\right) + O\left(|x|^{-\frac32}\right)\quad \mbox{as $|x|\to\infty$},
\end{equation}
see e.g. \cite[Theorem 4]{krasikov}. This field, called in this way in honor of Berry who introduced it in the seminal paper \cite{Berry}, can be seen as a universal Gaussian field
emerging as the local scaling limit of a number of random fields on two-dimensional manifolds, see e.g. \cite{xxx} and the references therein. 
It is widely regarded as a popular model for the Laplacian eigenfunctions (with large eigenvalue $t^2$) of classically chaotic billiards, hence
its importance in quantum mechanics. Indeed, integrating $\varphi(B_x)$ over $tD$ is the same as integrating $\varphi(B_{tx})$ over the fixed domain $D$, after a change of variable.
As we will see, our Theorem \ref{thm:main} will allow to prove the Gaussian fluctuations of $\int_D\varphi(B_{tx})dx$ for many functionals that were never investigated in the literature, in particular for the cases $R=2$ and $R=1$ in (\ref{rate}), enriching the knowledge on the geometric properties of the Berry's random wave model. Indeed, to the best of our knowledge, so far Gaussian fluctuations were proved only for (\ref{Yt}) under the assumption $R\ge4$ (see e.g. \cite{Notarnicola}), or for nodal set volumes (see \cite{NPR}), where the latter cannot be expressed in the form (\ref{Yt}).\\ 

If we compare (\ref{pasclassique}) with (\ref{classique}), it is like we have $\beta=\frac12$ and that we had replaced the slowly varying function $L$ in (\ref{classique})  by the bounded and oscillatory function $\cos(\cdot-\frac{\pi}4)$.
As we will see, this replacement is all but insignificant in the presence of long memory.
Indeed, taking for example $D$ compact such that {$D=\overline{\mathring{D}}$ and with smooth $\partial D$ and non-vanishing Gaussian curvature}, if we apply our main result (Theorem \ref{thm:main} below) for $R\in \{1,2,4\}$ and Breuer-Major theorem for $R\ge5$, we obtain (recalling that $R$, resp. $R'$, denotes the Hermite rank, resp. the second Hermite rank, of the \textbf{non-odd} function $\varphi$):
\begin{equation}\label{rate}
\sigma_t^2 \asymp \left\{\begin{array}{lll}
t^3&&\mbox{if $R=2$ or  ($R=1$ and $R'=2$)}\\
t^2\log(t)&&\mbox{if $R=4$ or  ($R=1$ and $R'=4$)}\\
t^2&&\mbox{if $R\ge5$ or  ($R=1$ and $R'\ge5$)}
\end{array}
\right.
\end{equation} 
and
\begin{equation}\label{example}
\frac{Y_t-m_t}{\sigma_t}\overset{\rm law}{\to} N(0,1)\quad \mbox{as $t\to\infty$}.
\end{equation}
Indeed, Gaussian fluctuations (\ref{example}) and asymptotics (\ref{rate}) follow directly from (\ref{pasclassique}) and the Breuer-Major Theorem \ref{BM} when $R\geq 5$, because in this case we have $\int_{\R^d}|C(x)|^Rdx<\infty$.
The case $R=4$ is comparable with the situation (ii) in the model (\ref{classique}), since $\beta=\frac12=\frac{d}R$; Gaussian fluctuations (\ref{example}) and asymptotic (\ref{rate}) displaying a logarithmic correction are then not surprising, since 
in agreement with what is usually
observed in critical cases.
In constrast, the fact that we still have Gaussian fluctuations in (\ref{example}) with unconventional rate when $R\in\{{1,2}\}$
is very surprising.
Indeed, since $\beta=\frac12<1\le\frac{d}{R}$ in this case, we are in the long-memory case
where non-Gaussian fluctuations are usually the rule, see (iii). 
This last case turns out to be the most important, since $R=1$ and $R=2$ are the most common values in applications.

As we will see, what we just described (Gaussian fluctuations in a long memory framework) for BRWM is not an isolated phenomenon. In fact, we are going to state and prove Theorem \ref{thm:main} (our main result), which not only explains in a clear way the phenomenon observed for BRWM, but also bring to light an easy-to-check condition on the spectral measure of a general Gaussian field $B$ (see (\ref{spectralcond}) below) to imply Gaussian fluctuations for $Y_t$.

We conclude this section by emphasing that Gaussian fluctuations in presence of long memory have already been observed in the literature in other (non-Euclidean) contexts, in particular for integral functionals of random Laplace eigenfunctions on the sphere $\mathbb{S}^2$, see \cite{MW} for more details.

\subsection{Main result}
In order to state Theorem \ref{thm:main}, we have to introduce a certain number of further quantities and notations.
We continue to let $B$, $\varphi$ and $Y_t$ be as described in (\ref{C}), (\ref{varphii}) and (\ref{Yt}) respectively,
and we  recall the Hermite decomposition (\ref{hermitedecomp}) of $\varphi$ defining its Hermite rank $R$ and its second Hermite rank $R'$.

Fix $t> 0$ and $q\geq R$, 
and  
set 
\begin{equation}\label{yqt}
Y_{q,t} =  \int_{tD} H_q(B_x)dx.
\end{equation} 
Using  (\ref{hermitedecomp}) , we immediately get that
\begin{equation}
	\label{chaoticdecomp}
	Y_t = \E[Y_t]+ \sum_{q=R}^\infty a_q\,Y_{q,t},\quad t\geq 0.
\end{equation}
Moreover,
a direct computation (making use of the isometry properties of Hermite polynomials, see, e.g.,  \cite[Section 1.4]{bluebook}) yields that 
$$
{\rm Var}(Y_{q,t})= q!t^d v_{q,t},$$ where
		\begin{equation*}
			v_{q,t}=\int_{\R^d}C(z)^q\,g_D\left(\frac{z}{t}\right)dz=\int_{\{|z|\le {\rm diam}(D)t\}}C(z)^q\,g_D\left(\frac{z}{t}\right)dz\geq 0,
		\end{equation*}
with $g_D(x)$ the {\bf covariogram} of $D$ at $x\in\R^d$, defined as the Lebesgue measure of $D\cap (x+D)$, and ${\rm diam}(D):=\sup\{|x-y|:x,y\in D\}<\infty$ since $D$ is compact.
When $C\in L^q(\R^d)$, we deduce by dominated convergence (using that $\|g_D\|_\infty\leq {\rm Vol}(D)$ and $g_D(\frac{z}{t})\to {\rm Vol}(D)$ as $t\to\infty$ for all fixed $z\in\R^d$)
that $$v_{q,t}\to {\rm Vol}(D)\int_{\R^d}C(z)^qdz\geq 0.$$

Since our field $B$ is stationary and its covariance function $C$ is continuous,
Bochner's theorem yields the existence of a finite measure $G$ on $\R^d$ such that
\begin{equation}\label{fourier}
C(x) = \int_{\R^d} e^{i\langle \lambda,x\rangle} G(d\lambda),
\end{equation}
with $\langle\cdot,\cdot\rangle$ the usual scalar product on $\R^d$.
The measure $G$ is called the {\bf spectral measure} of $B$, and it will be our gateway  towards the Fourier analysis techniques
developed in the sequel.

At this stage, let us make a further assumption on $B$, by supposing that it is also {\bf isotropic}.
In our framework, this is equivalent to suppose that the quantity $C(x)$ only depends on the norm $|x|$, namely that there is a function $\rho:\R_+\to\R$ such that 
$$C(x)=\rho(| x|),\quad x\in\R^d.$$

Now, set 
\begin{equation}\label{mu}
\mu(s)=G(\{|x|\le s\}), \quad s\in (0,\infty).
\end{equation}
Since $\mu$ is increasing and bounded, it defines a finite measure on $\R_+$, called the {\bf isotropic spectral measure} of $B$.
Because $C(x)$ only depends of $|x|$, we can write, with $w_d=\int_{S^{d-1}}d\xi$ ($S^{d-1}$ being the unit sphere of $\R^d$): 
\begin{eqnarray}
\rho(r)&=&\frac1{w_d}\int_{S^{d-1}} C(r\xi)d\xi = \frac1{w_d}\int_{S^{d-1}} d\xi \int_{\R^d}e^{i\langle\lambda,r\xi\rangle}G(d\lambda)\notag\\
&=& \int_{\R^d} b_d(r|\lambda|) G(d\lambda) = \int_0^\infty b_d(rs)\mu(ds),\label{rho}
\end{eqnarray}
where
\begin{equation}\label{bd}
b_d(|\lambda|) = \frac1{w_d}\int_{S^{d-1}} e^{i\langle\lambda,\xi\rangle}d\xi.
\end{equation}
By $d$-dimensional polar coordinates, we readily find (see \cite[page 815]{Schoenberg}) that
\begin{equation}\label{bd2}
b_d(r) = c_d\,r^{1-\frac{d}2}J_{\frac{d}2-1}(r), \quad r> 0,
\end{equation}
with $J_\nu$ denoting the Bessel function of the first kind of order $\nu$ (see Section \ref{besselstuff}) and $c_d> 0$ a constant depending only of $d$.

\medskip

We are at last in a position to state our main result, which we have decided to call the 
{\bf spectral central limit theorem},
because it leads to Gaussian fluctuations on the one hand (hence `central') and the main assumption we have to check is the spectral condition (\ref{spectralcond}) on the other hand (hence `spectral').

\begin{theorem}[Spectral CLT]\label{thm:main}
	Fix $d\geq 2$, let $B=(B_x)_{x\in\R^d}$ be a real-valued continuous centered Gaussian field  on $\R^d$, and
	assume that $B$ is stationary, isotropic and has unit-variance.
	Let $\varphi:\R\to\R$ be \textbf{not odd} and such that $\E[\varphi(N)^2]<\infty$ with $N\sim N(0,1)$, let $R$ (resp. $R'$) be the Hermite rank (resp. second Hermite rank) of $\varphi$, where $(R,R')\notin\{(1,3)\}\cup\{(2k+1,n):k\ge1, n\in\N\}$.
	Consider $Y_t$ defined by (\ref{Yt}), where $D\subset\R^d$  compact and ${\rm Vol}(D)>0$. Set $m_t=\E[Y_t]$ and $\sigma_t=\sqrt{{\rm Var}(Y_t)}>0$, and 
	recall the definition  (\ref{mu}) of the isotropic spectral measure $\mu$.
	Set
	\begin{equation}
		\label{var chaos bis}
		w_{q,t}=\int_{\{|z|\le t\}}C(z)^q dz,
	\end{equation}
	and assume that the following spectral condition holds
	\begin{equation}\label{spectralcond}
		\int_0^\infty s^{-\frac{d}{R}}\mu(ds)<\infty.
	\end{equation}
	Finally, when $R=2$ assume that $\left|\mathcal{F}[\mathbf{1}_{D}](x)\right|=O\left(\frac{1}{|x|^{d/2}}\right)$ as $|x|\rightarrow\infty$, and when $R=1$ assume that $|\mathcal{F}[\mathbf{1}_{D}](x)|=o\left(\frac{1}{|x|^{d/2}}\right)$ as $|x|\rightarrow\infty$, with $\mathcal{F}$ the Fourier transform (since we assumed $D$ to be compact, these two assumptions are satisfied for example when $D=\overline{\mathring{D}}$ and $\partial D$ is smooth with non-vanishing Gaussian curvature, see e.g. \cite{Brandolini}). Then we have:
	\begin{equation*}
		\sigma_t^2 \asymp \left\{\begin{array}{lll}
			t^dw_{R,t}&&\mbox{if $R$ even}\\
			t^dw_{R',t}&&\mbox{if $R=1$ and $R'\in\{2,4\}$}\\
			t^d&&\mbox{if $R=1$ and $R'\ge5$}
		\end{array}
		\right.
	\end{equation*} 
	and
	\begin{equation*}
		\frac{Y_t-m_t}{\sigma_t}\overset{\rm law}{\to} N(0,1)\quad \mbox{as $t\to\infty$}.
	\end{equation*}
\end{theorem}

\medskip

By a simple Fubini argument, we observe that the spectral condition (\ref{spectralcond})
is equivalent to
\begin{equation}\label{spectralcond2}
	\int_0^\infty s^{-1-\frac{d}{R}}\mu(s)ds<\infty.
\end{equation}
When $\rho(r)=r^{-\beta}L(r)$ with $\beta\in(0,\frac{d}{R})$ and $L$ slowly varying as $r\to\infty$,
we have
$\mu(s)\sim s^{\beta}L(s^{-1})$ as $s\to 0$, see \cite[Thm 1.4.3]{leonenko}.
We deduce that (\ref{spectralcond})-(\ref{spectralcond2}) is
not satisfied, which is of course perfectly consistent with the conclusion of the Dobrushin-Major noncentral limit theorem \cite{DM}.

Also, let us observe that the assumption "$\varphi$ is not odd" is equivalent to "$\exists$ $k\ge1$ such that $a_{2k}\neq0$", and so is needed only in the case where $R=1$ and $R'\ge5$. 
Moreover, the cases not covered by our Theorem \ref{thm:main} are $R\ge 3$ odd, $(R,R')=(1,3)$ and $\varphi$ odd. In these situations, very peculiar things can happen.
 Consider for example $\varphi(x)=x$, which gives $Y_t=\int_{tD}B_x dx$. 
Using  (\ref{var chaos}), (\ref{fourier}) and Fubini we have, with $\mathcal{F}$ the Fourier transform,
\begin{equation*}
	{\rm Var}(Y_t)=\int_{\R^d}C(z)g_{tD}(z)dz=\int_{\mathbb{R}^d}\mathcal{F}[g_{tD}](\lambda)G(d\lambda).
\end{equation*}
Using that $g_{tD}=\mathbf{1}_{tD}\ast\mathbf{1}_{-tD}$ we obtain
\begin{equation}\label{extrank1}
	{\rm Var}(Y_t)={\rm const}\int_{\mathbb{R}^d}\frac{t^{d}}{|\lambda|^d}|t\lambda|^d|\mathcal{F}[\mathbf{1}_{D}](t\lambda)|^2 G(d\lambda) .
\end{equation}
In particular, for $D=\{|z|\le1\}\subset \R^d$ and Berry's Random Wave Model $B=(B_x)_{x\in\R^2}$ we get (using (\ref{foufou}))
$$\sigma_t^2={\rm Var(Y_t)}={\rm const} \, t^2 J^2_{1}(t);$$
although $Y_t$ is Gaussian the conclusion of Theorem \ref{thm:main} cannot hold in this case, the
limit of $\frac{Y_t-m_t}{\sigma_t}$ being ill-defined due to the fact that ${\rm Card}\{t:\,\sigma_t=0\}\cap [T,\infty)=+\infty$ for all $T>0$.
\medskip

To understand the significance of our spectral CLT, let us go back to BRWM and explain
how Theorem \ref{thm:main} together with Breuer-Major theorem allow to prove (\ref{example}).
Since $C(x)=\rho(|x|)=J_0(|x|)$, it follows immediately from the representation (\ref{rho})
that the isotropic spectral measure associated with BRWM $B=(B_x)_{x\in\R^2}$ is $\mu= \delta_1$, with $\delta_1$ 
the Dirac mass at 1. 
In particular, the spectral condition (\ref{spectralcond}) is obviously satisfied, whatever the value of $R$.
The convergence (\ref{example}) thus follows from the Breuer-Major Theorem \ref{BM} (see also (\ref{pasclassique})) if $R\ge5$ and from Theorem \ref{thm:main} in all the other cases in (\ref{rate}).

\subsection{Possible natural extensions of Theorem \ref{thm:main}}

As a natural extension of the present work, it would be interesting to study the \textit{joint convergence} associated with Theorem \ref{thm:main}. More precisely, taking $D_1,\dots,D_n$ compact domains in $\R^d$, 
is it possible to identify conditions that resemble those in Theorem \ref{thm:main}
ensuring that the random vector $$\left(\int_{tD_1}\varphi(B_x)dx, \dots, \int_{tD_n}\varphi(B_x)dx\right)$$ converges, after proper normalization, to a Gaussian random vector? In other words, can we prove a {\it multivariate} spectral central limit theorem?
Such an extension is not immediate, because the spectral condition (\ref{spectralcond}) alone looks too general to capture the asymptotic behavior of the covariances between the components of the random vector. A partial answer to this question (covering also Berry's model) will be given in the forthcoming work \cite{M} by the first author.
		Note that a multivariate result was obtained in the particular case of the nodal length of Berry's model restricted to a finite collection of smooth compact domains $D_1,\ldots,D_n$ of the plane
		in the recent paper \cite{PV}. 
		
	A further stronger generalization of Theorem \ref{thm:main} would be to prove a {\it functional} spectral central limit theorem. 
	This problem has been investigated in the particular case of nodal set volumes in \cite{NPV} by ad hoc techniques. 
	In our general framework, 
	we believe however that proving such a functional extension would require novel ideas that go beyond the scope of the present paper. 

\subsection{Plan of the paper}
Apart from Section \ref{Examples-sec} that illustrates a use of Theorem \ref{thm:main}, the rest of the paper is fully devoted to the proof of this latter. More precisely, after a few needed preliminaries given in Section \ref{preliminaries}, 
\begin{enumerate}
\item in Section \ref{step1} we will first consider the situation where the Hermite rank $R$ of $\varphi$ is even and bigger or equal than $4$.
As we will see, in this case we have that $\int_{tD} |C(x)|^Rdx = O(\log t)$, meaning that we are somehow in the `domain of attraction' of the Breuer-Major theorem \ref{BM}.;
\item we will then deal with the remaining cases, namely $R=2$ in Section \ref{step2} and ($R=1$, $\varphi$ non-odd and $R'\neq3$) in Section \ref{rank1section}.
We regard this part as the most important contribution of this paper, especially since among the Hermite ranks, $R=1$ and $R=2$ are the most common values encountered in practice. 
They turn out to be also the most difficult cases. To deal with them, we will have to introduce novel ideas, by making heavy use of Fourier techniques in a way that, to the best of our knowledge, has never been introduced before this work.
\end{enumerate}

\section{A few preliminaries for the proof of Theorem \ref{thm:main}}\label{preliminaries}

\subsection{Well-posedness of $Y_t$}
 
 The following proposition explains why the random variable $Y_t$ defined by (\ref{Yt}) is well-defined for all $t$.

\begin{proposition}\label{wellposed}
Let $B=(B_x)_{x\in\R^d}$ be a real-valued continuous centered Gaussian field on $\R^d$, assumed to be stationary and to have unit-variance.
Let $D\subset \R^d$ be compact with ${\rm Vol}(D)>0$ and let $\varphi:\R\to\R$ be such that $\E[\varphi(N)^2]<\infty$ with $N\sim N(0,1)$.
Consider the Hermite expansion of $\varphi$ given by (\ref{hermitedecomp}), and set
$$\varphi_n = \mathbb{E}[\varphi(N)] + \sum_{q=R}^n a_qH_q,\quad n\geq R.$$
For each fixed $t>0$, the sequence $\int_{tD}\varphi_n(B_x)dx$, $n\geq R$, is a.s. well-defined and converges in $L^2(\Omega)$. The limit is noted $Y_t$ and we write, possibly with a slight abuse of notation:
$$
Y_t=\int_{tD}\varphi(B_x)dx.
$$
\end{proposition}
\noindent
{\it Proof}. 
That $\int_{tD}\varphi_n(B_x)dx$ is a.s. well-defined is because $tD$ is compact, $\varphi_n$ is a polynomial and $B$ is continuous.
For any $n,m\geq R$ and $q\in\N$, we have (since $|C(z)|\le C(0)=1$, $a_0=\E[\varphi(N)]$ and $a_q=0$ if $q\in\{1,\dots, R-1\}$)
\[
\left|\mathbf{1}_{[0,n\wedge m]}(q)\int_{(tD)^2}q!a_q^2 C(x-y)^qdxdy\right|\le {\rm Vol}(tD)^2 q!a_q^2,
\]
with $\sum_{q=0}^\infty{\rm Vol}(tD)^2 q!a_q^2={\rm Vol}(tD)^2 \E[\varphi(N)^2]<\infty$ by assumption. Then, by dominated convergence we obtain as $n,m\rightarrow\infty$
\begin{eqnarray*}
&&\E\left[
\int_{tD}\varphi_n(B_x)dx \int_{tD}\varphi_m(B_y)dy
\right]
=\sum_{q=0}^{n\wedge m}q!a_q^2 \int_{(tD)^2}C(x-y)^qdxdy\\
&\to&\mathbb{E}[\varphi(N)]^2{\rm Vol}(tD)^2+\sum_{q=R}^\infty q!a_q^2 \int_{(tD)^2} C(x-y)^qdxdy.
\end{eqnarray*}
That is, $\left\{ \int_{tD}\varphi_n(B_x)dx\right\}_{n\geq R}$ is an $L^2(\Omega)$-Cauchy sequence, and the
desired conclusion follows.
\qed

\subsection{Bessel functions and Fourier transform of the indicator of the unit ball}\label{besselstuff}

The Bessel function $J_\nu$, $\nu\geq 0$, is defined by the series
$$
J_\nu(t)=\left(\frac{t}{2}\right)^\nu\,\,\sum_{j=0}^\infty (-1)^j \frac{(t^2/4)^j}{j!\Gamma(j+\nu+1)},\quad t\in\R,
$$
and it is a solution to the ODE
$$
t^2 J''_\nu(t) + t J'_\nu(t) + (t^2-\nu^2)J_\nu(t)=0.
$$

It satisfies the classical Schl\"afli's representation
\begin{equation}\label{schl}
J_\nu(t) = \frac1\pi \int_0^\pi \cos(\nu\theta-t\sin\theta)d\theta - \frac{\sin(\nu\pi)}{\pi}\int_0^\infty e^{-\nu\theta - t\sinh \theta}d\theta,
\end{equation}
as well as the
Mehler-Sonine formula
\begin{equation}\label{mehler}
J_\nu(t) = \frac{(t/2)^\nu}{\sqrt{\pi}\,\Gamma(\nu+\frac12)}\int_{-1}^1 e^{its}(1-s^2)^{\nu-\frac12}ds.
\end{equation}

Schl\"afli's representation (\ref{schl}) immediately implies that $J_\nu$ is bounded on $\R$. This, combined with an inequality found e.g. in \cite{krasikov} (see also the references therein), leads to
\begin{equation}\label{boundBessel}
\sup_{t\in\R_+} \sqrt{t}|J_\nu(t)|<\infty\quad\mbox{for all $\nu\geq 0$},
\end{equation}
a simple property that we will use several times in the forthcoming proof of Theorem \ref{thm:main}.

\medskip

The Fourier transform of the indicator of the unit ball is given by
$$
\mathcal{F}\left[{\bf 1}_{\{|x|\le1\}}\right](y)=\int_{\{|x|\le1\}} e^{i\langle x,y\rangle}dx.
$$
Since $\mathcal{F}\left[{\bf 1}_{\{|x|\le1\}}\right] $ is rotationally symmetric, we can write, with $\alpha_d$ the volume of the unit ball $\{|x|\le1\}\subseteq \R^d$, 
\begin{eqnarray}
	\mathcal{F}\left[{\bf 1}_{\{|x|\le1\}}\right](y) &=&\mathcal{F}\left[{\bf 1}_{\{|x|\le1\}}\right]((0,\ldots,0,|y|))\notag\\
	&=& \int_{-1}^1 (1-x_d^2)^{\frac{d-1}{2}} \alpha_{d-1} e^{ix_d|y|}dx_d\notag\\
	&=& \alpha_{d-1}\sqrt{\pi}\,\Gamma\left(\frac{d+1}2\right)2^\nu\,|y|^{-\frac{d}2}J_{d/2}(|y|),\label{foufou}
\end{eqnarray}
the last equality being a consequence of (\ref{mehler}). 

\subsection{Reduction to the $R$th chaos}

In the short memory case, that is when $C\in L^R(\R^d)$, the Breuer-Major theorem \ref{BM} yields Gaussian fluctuations for $Y_t$.
In its modern proof given by \cite{NPP} (see also \cite[Chapter 7]{bluebook}) the chaotic expansion of $Y_t-m_t$ is considered, namely
\begin{equation}\label{eur}
Y_t-m_t = \sum_{q=R}^\infty a_q Y_{q,t},
\end{equation}
and the proof goes as follows. 
It is first shown that $t^{-d}\sigma_t^2\to \sigma^2$ in (\ref{varBM}) by means of the isometry property of Hermite polynomials.
Then, it is proved using the Fourth Moment Theorem (see \cite[Theorem 5.2.7]{bluebook}) that 
$t^{-d/2}Y_{q,t}\overset{\rm law}{\to} N(0,\sigma^2_q)$ for {\it all} $q\geq R$
from which it is deduced that $t^{-\frac d2}(Y_t-m_t)\overset{\rm law}{\to} N(0,\sum_{q=R}^\infty a_q^2\sigma_q^2)$ thanks to \cite[Theorem 5.2.7]{bluebook}.
In particular, we observe that no term is asymptotically dominant in (\ref{eur}): they all contribute to the limit.

As the following result will show, the situation is totally opposite in the critical and long memory cases (when $R$ is even): here, it is the term $Y_{R,t}$ alone which is
responsible of the limit.

\begin{proposition}[Reduction to the $R$th chaos]\label{Prop5.4}
	Let $B=(B_x)_{x\in\R^d}$ be a real-valued continuous centered Gaussian field  on $\R^d$, and
	assume that $B$ is stationary and has unit-variance (note that $d\ge2$ and isotropy are not required here).
	Let $\varphi:\R\to\R$ be such that $\E[\varphi(N)^2]<\infty$ (with $N\sim N(0,1)$) and have Hermite decomposition (\ref{hermitedecomp}) and Hermite rank $R$,
	and consider $Y_t$ defined by (\ref{Yt}), where $D\subset\R^d$  compact and ${\rm Vol}(D)>0$. Set $m_t=\E[Y_t]$, $\sigma_t=\sqrt{{\rm Var}(Y_t)}>0$,  
	and recall the definition (\ref{C}) of $C(x)$ and the definition (\ref{yqt}) of $Y_{q,t}$.
	
	If  $R$ is \textbf{even},  if $t^{-d}{\rm Var}(Y_{R,t})\to\infty$ and if $\int_{\R^d}|C(x)|^M dx<\infty$ for some $M\geq R+1$
	then, as $t\to\infty$, we have $\sigma_t^2  \asymp {\rm Var}(Y_{R,t})$ and
	\begin{equation}\label{iff}
		Y_{R,t}/\sqrt{{\rm Var}(Y_{R,t})}\overset{\rm law}{\to} N(0,1)\quad \Longrightarrow\quad\frac{Y_t-m_t}{\sigma_t}\overset{\rm law}{\to} N(0,1).
	\end{equation}
\end{proposition}
\noindent
{\it Proof}. We have
\[
{\rm Var}(Y_{q,t})=\int_{(tD)^2}\mathbb{E}[H_q(B_y)H_q(B_x)]dxdy = q!\int_{(tD)^2}C(x-y)^q dxdy.
\]
Applying the change of variable $z=y-x$ and then Fubini, we obtain that
$${\rm Var}(Y_{q,t}) = q!t^d v_{q,t},$$ where
\begin{equation}
	\label{var chaos}
	v_{q,t}=\int_{\{|z|\le {\rm diam}(D)t\}}C(z)^q\,g_D\left(\frac{z}{t}\right)dz,
\end{equation}
where ${\rm diam}(D)=\sup{\{|x-y|:x,y\in D\}}<\infty$ (because $D$ is compact) and $g_D(x)$ is the covariogram of $D$ at $x\in\R^d$, that is, $g_D(x)$ is the Lebesgue measure of $D\cap (x+D)$.
Also, set
\begin{equation}
	\label{var chaos2}
	\widetilde{v}_{q,t}=\int_{\{|z|\le {\rm diam}(D)t\}}|C(z)|^q\,g_D\left(\frac{z}{t}\right)dz,
\end{equation}
and observe that $v_{R,t}=\widetilde{v}_{R,t}>0$ since $R$ is even.

For any $q>R$, we have 
\begin{equation}
	\label{var ratio}
	\frac{{\rm Var}(Y_{q,t})}{{\rm Var}(Y_{R,t})}=\frac{q!v_{q,t}}{R!v_{R,t}}.
\end{equation}
Applying Cauchy-Schwarz inequality we obtain, for $q>R$,
\begin{eqnarray*}
	\frac{|v_{q,t}|}{v_{R,t}}&&\leq \frac{\int_{\{|z|\le {\rm diam}(D)t\}}|C(z)|^qg_D(\frac{z}{t})dz}{v_{R,t}}\\
	&&=\frac{\int_{\{|z|\le {\rm diam}(D)t\}}|C(z)|^{q-\frac{R}{2}}|C(z)|^{\frac{R}{2}}g_D\left(\frac{z}{t}\right)dz}{v_{R,t}} \\
	&&\le \left(\frac{\int_{\{|z|\le {\rm diam}(D)t\}}|C(z)|^{2q-R}g_D\left(\frac{z}{t}\right)dz}{v_{R,t}}\right)^{1/2} =\left(\frac{\widetilde{v}_{2q-R,t}}{v_{R,t}}\right)^{1/2}.
\end{eqnarray*}
Applying Cauchy-Schwarz again, but this time with $2q-R>R$ instead of $q>R$, we obtain 
\[
\left(\frac{\widetilde{v}_{2q-R,t}}{v_{R,t}}\right)^{1/2}\le\left(\frac{\widetilde{v}_{4q-3R,t}}{v_{R,t}}\right)^{\frac{1}{4}}.
\]
By iterating the process, we get,  for every $n\ge3$
\[
\frac{|v_{q,t}|}{v_{R,t}}\le\left(\frac{\widetilde{v}_{2q-R,t}}{v_{R,t}}\right)^{1/2}\le\left(\frac{\widetilde{v}_{4q-3R,t}}{v_{R,t}}\right)^{\frac{1}{4}}\le \dots\le \left(\frac{\widetilde{v}_{R+2^n(q-R),t}}{v_{R,t}}\right)^{\frac{1}{2^n}}.
\]
When $q>R$, we have $R+2^n(q-R)\geq 2^n$, so we may and will choose $n$ large enough so that $R+2^n(q-R)\geq M$  for all $q>R$ (recall from the statement of Proposition \ref{Prop5.4} that $M$ is an integer supposed to be such that $\int_{\R^d}|C(x)|^Mdx<\infty$). 
Using $|C(x)|\le C(0)=1$ and $g_D(x)\leq {\rm Vol}(D)$ for all $x\in\mathbb{R}^d$, we deduce that $\widetilde{v}_{R+2^{n}(q-R),t}\le {\rm Vol}(D)\int_{\R^d}\big|C(x)\big|^Mdx$. Combining all these facts together,
we finally get that
\begin{equation}\label{kivabien2}
	\frac{|v_{q,t}|}{v_{R,t}}\leq \left({\rm Vol}(D)\int_{\R^d}|C(x)|^Mdx\right)^{\frac{1}{2^n}}\, v_{R,t}^{-\frac{1}{2^n}}\quad\mbox{for all $q>R$}.
\end{equation}

Since ${\rm Var}(Y_t)=\sum_{q=R}^\infty a_q^2\,{\rm Var}(Y_{q,t})$, we deduce
\begin{equation}\label{kivabien}
	\left|\frac{{\rm Var}(Y_t)}{{\rm Var}(Y_{R,t})}-a^2_R\right|\leq \frac{1}{R!}\,\left({\rm Vol}(D)\int_{\R^d}|C(x)|^Mdx\right)^{\frac{1}{2^n}}\,\left(\sum_{q=R+1}^\infty a_q^2q!\right)\,v_{R,t}^{-\frac{1}{2^n}},
\end{equation}
from which it comes that ${\rm Var}(Y_t)\sim a^2_R {\rm Var}(Y_{R,t})$ as $t\to\infty$, since $v_{R,t}=\frac1{R!}t^{-d}{\rm Var}(Y_{R,t})\to\infty$ by assumption.

To conclude the proof of Proposition \ref{Prop5.4}, 
it remains to prove 
(\ref{iff}). 
Using the decomposition
\begin{eqnarray*}
	&&\frac{Y_t-m_t}{\sqrt{{\rm Var}(Y_t)}} - {\rm sgn}(a_R)\frac{Y_{R,t}}{\sqrt{{\rm Var}(Y_{R,t})}}\\
	&=&\frac{{\rm sgn}(a_R)(Y_t-m_t-a_RY_{R,t})}{a_R\sqrt{{\rm Var}(Y_{R,t})}} + \frac{Y_{t}-m_t}{\sqrt{{\rm Var}(Y_{t})}}\left\{1-
	\frac{1}{|a_R|}\sqrt{\frac{{\rm Var}(Y_{t})}{{\rm Var}(Y_{R,t})}}
	\right\}
\end{eqnarray*}
we get that
\begin{eqnarray}\notag
	&&\mathbb{E}\left[\left(\frac{Y_t-m_t}{\sqrt{{\rm Var}(Y_t)}} - {\rm sgn}(a_R)\frac{Y_{R,t}}{\sqrt{{\rm Var}(Y_{R,t})}}\right)^2\right]\\
	&\leq &2\frac{\mathbb{E}[(Y_t-m_t-a_RY_{R,t})^2]}{a_R^2{\rm Var}(Y_{R,t}) }+ 2\,\left(1-
	\frac{1}{|a_R|}\,\sqrt{\frac{{\rm Var}(Y_{t})}{{\rm Var}(Y_{R,t})}}
	\right)^2\label{rrhs}.
\end{eqnarray}
Since $\mathbb{E}\left[\left({Y_t}-m_t-a_R Y_{R,t}\right)^2\right]=\sum_{q=R+1}^\infty a^2_q\,{{\rm Var}(Y_{q,t})}$, we deduce from (\ref{var ratio}) and  (\ref{kivabien2}) that
\begin{eqnarray*}
	&&
	\frac{\mathbb{E}[(Y_t-m_t-a_RY_{R,t})^2]}{a_R^2{\rm Var}(Y_{R,t}) }\\
	&
	\leq& 
	\frac{1}{R!a^2_R} \left({\rm Vol}(D)\int_{\R^d}|C(x)|^Mdx\right)^{\frac{1}{2^n}}\,\left(\sum_{q=R+1}^\infty q!a^2_q \right) v_{R,t}^{-\frac{1}{2^n}},
\end{eqnarray*}
and this tends to zero as $t\to\infty$.
By plugging this into (\ref{rrhs}) and taking into account that (\ref{kivabien}) holds, we deduce that 
$$
\frac{Y_t}{\sqrt{{\rm Var}(Y_t)}} - {\rm sgn}(a_R)\frac{Y_{R,t}}{\sqrt{{\rm Var}(Y_{R,t}})}\overset{L^2(\Omega)}{\to} 0\quad\mbox{as $t\to\infty$},
$$
from which the implication (\ref{iff}) now follows easily.\qed

\subsection{Elements of Malliavin calculus and the Fourth Moment theorem}\label{fmt}
To obtain the $N(0,1)$ distribution in the limit in Theorem \ref{thm:main}, we rely on the Fourth Moment Theorem of Nualart and Peccati
(see \cite[Theorem 5.2.7]{bluebook}). Before stating and reformulating it in our framework, we start with some notions on Malliavin calculus. 
For all the missing details we refer to \cite{bluebook} and \cite{N}.

\subsubsection{The Wiener-It\^o integral}Let $B=(B_x)_{x\in\R^d}$ be a real-valued continuous centered Gaussian field  on $\R^d$, and
assume that $B$ is stationary with unit-variance.
Define  
\[
\mathcal{H}=\overline{{\rm span}{\{B_x, \,x\in\R^d\}}}^{L^2(\Omega)}.
\]
Since $\mathcal{H}$ is a real, separable Hilbert space, there is an isometry $\Phi: \mathcal{H} \rightarrow L^2(\R_+)$. If we set $e_x:=\Phi(B_x)$ for every $x\in\R^d$, we have 
\[
\E[B_xB_y]=C(x-y)=\langle e_x,e_y\rangle_{L^2(\R_+)}.
\]
Consider now the Gaussian noise $W=\{W(h), h\in L^2(\R_+)\}$, i.e. a family of centered Gaussian random variables with covariance given by
\[
\E[W(h)W(g)]=\langle h,g \rangle_{L^2(\R_+)}.
\]
Since in this paper we are only interested in distributions, we can assume without loss of generality that $(B_x)_{x\in\R^d}=(W(e_x))_{x\in\R^d}$. 

For $q\ge1$, we also define the $q$th Wiener chaos as the linear subspace of $L^2(\Omega)$ generated by $\{H_q(W(h)), h\in L^2(\R_+)\}$. For every $h\in L^2(\R_+)$ with $\norm{h}_{L^2(\R_+)}=1$, we define the $q$th \textbf{Wiener-It\^o integral}
\[
I_q(h^{\otimes q})=H_q(W(h)),
\]
where $h^{\otimes q}:\R_+^q\rightarrow \R$ is defined by $$h^{\otimes q}(x_1,\dots,x_q)=\prod_{r=1}^{q}h(x_r).$$
Note that the definition of $I_q$ can be extended to every function in the space $L^2_s(\R_+^q)$ of symmetric functions in $L^2(\R_+^q)$ so that $I_q:L^2_s(\R_+^q)\rightarrow L^2(\Omega)$ is a linear map, because ${\rm span}\{h^{\otimes q}, h\in L^2(\R_+)\}$ is dense in $L^2_s(\R_+^q)$ (see, e.g., \cite{Floret}).

\subsubsection{Contractions and the Fourth Moment Theorem}
For $q\in\N$, $r\in\{1,\dots,q-1\}$ and $h_1,h_2$ symmetric functions with unit norm in $L^2(\R_+)$, we can define the $r$th \textbf{contraction} of $h_1^{\otimes q}$ and $h_2^{\otimes q}$ as the (non-symmetric) element of $L^2(\R_+^{2q-2r})$ given by
\[
h_1^{\otimes q}\otimes_r h_2^{\otimes q}=\langle h_1,h_2 \rangle_{L^2(\R_+)}^r \, h_1^{\otimes q-r}\otimes h_2^{\otimes q-r}.
\]  
Here again this definition can be extended (taking the closure in $L^2(\R_+^q)$) to every $h_1,h_2$ in $L^2_s(\R_+^q)$. We will denote the norm in the space $L^2(\R_+^{q})$ as $\norm{ \cdot}_q$. We can finally state the celebrated \textbf{Fourth Moment Theorem}, proved in \cite{NualartPeccati} by Nualart and Peccati.
\begin{theorem}[Fourth Moment Theorem]
	Fix $q\ge2$, consider $(h_t)_{t>0}\subset L^2_s(\R_+^q)$ and assume that $\E[I_q(h_t)^2]\rightarrow 1$ as $t\rightarrow\infty$. Then the following assertions are equivalent:
	\begin{itemize}
		\item $I_q(h_t)$ converges in distribution to a standard Gaussian $N\sim N(0,1)$.
		\item $\E[I_q(h_t)^4]\rightarrow 3=\E[N^4]$, where $N\sim N(0,1)$.
		\item $\norm{h_t\otimes_rh_t}_{2q-2r}\rightarrow 0$ as $t\rightarrow\infty$, for all $r=1,\dots,q-1$.
	\end{itemize}
\end{theorem}
An important consequence of the previous result (in our framework) is the following.
\begin{theorem}
	\label{fourth moment theorem}
	Let $B=(B_x)_{x\in\R^d}$ be a real-valued continuous centered Gaussian field  on $\R^d$, and
	assume that $B$ is stationary and has unit-variance. Assume that $D\subset \R^d$ is compact with ${\rm Vol}(D)>0$.
	Recall the definition (\ref{C}) of $C$.
	Fix also $q\geq 2$, 
	recall the definition (\ref{yqt}) of $Y_{q,t}$ and
	assume ${\rm Var}(Y_{q,t})>0$ for all $t$ large enough.
	If we have, for any $r\in\{1,\dots,q-1\}$, that
	\begin{equation}
		\label{contractions}
		\frac{t^d}{{\rm Var}^2(Y_{q,t})}\int_{\{|u|\le diam(D)t\}^3}|C(x)|^r|C(y)|^r|C(z)|^{q-r}|C(x+y+z)|^{q-r}dxdy dz
	\end{equation}
	converges to $0$ as $t\to\infty$, 
	then $Y_{q,t}/\sqrt{{\rm Var}(Y_{q,t})}\overset{law}{\rightarrow}N(0,1)$ as $t\rightarrow\infty$.
\end{theorem}
\begin{proof}
	Let us first write $Y_{q,t}$ as a $q$-th multiple Wiener-It\^o integral with respect to $B$:
	$$Y_{q,t}=\int_{tD}H_q(B_x)dx=\int_{tD}H_q(W(e_x))dx=\int_{tD}I_q(e_x^{\otimes q})dx=I_q(f_{t,q}),$$ 
	where
	\[
	f_{t,q}=\int_{tD}e_x^{\otimes q}dx.
	\] 
	Since by definition
	\[
	f_{t,q}\otimes_rf_{t,q}=\int_{tD}\int_{tD}C^r(x-y)e_x^{\otimes q-r}e_y^{\otimes q-r}dxdy,
	\]
	we obtain
	\begin{eqnarray*}
		&&\norm{f_{t,q}\otimes_r f_{t,q}}_{2q-2r}^2\\		
		&=&\int_{(tD)^4}C(x_1-x_3)^rC(x_2-x_4)^r C(x_1-x_2)^{q-r}\\
		&&\hskip4cm \times C(x_3-x_4)^{q-r}dx_1dx_2dx_3 dx_4.
	\end{eqnarray*}
	Applying the change of variable $x=x_3-x_1$, $y=x_2-x_4$, $z=x_1-x_2$, $a=x_4$ (whose  Jacobian is equal to 1) and using the symmetry of $C$, we get that
	\begin{eqnarray*}
		&&\norm{f_{t,q}\otimes_r f_{t,q}}_{2q-2r}^2\le {\rm Vol}(D)\\ &&t^d\int_{\{|u|\le {\rm diam}(D)t\}^3}\!\!\!\!\!\!\!|C(x)|^r|C(y)|^r |C(z)|^{q-r}|C(x+y+z)|^{q-r}dxdydz.
	\end{eqnarray*}
	The Fourth Moment Theorem asserts
	that $Y_{q,t}/\sqrt{{\rm Var}(Y_{q,t})}$ converges in distribution to  $N(0,1)$ if (and only if)
	$
	\norm{f_{t,q}\otimes_r f_{t,q}}_{2q-2r}/{\rm Var}(Y_{q,t})\rightarrow 0
	$ for all $r=1,\dots,q-1$.
	The desired conclusion thus follows from (\ref{contractions}) and the previous bound for $\norm{f_{t,q}\otimes_r f_{t,q}}_{2q-2r}^2$. 
\end{proof}

\section{Proof of Theorem \ref{thm:main} when $R\geq 4$ even}\label{step1}

This section is devoted to the proof of Theorem \ref{thm:main} when $R\geq 4$ is even, which is equivalent to say that $\frac{d}{R}\leq \frac{d-1}2$ and $R$ even.
It represents the `easy' part of Theorem \ref{thm:main}.

To ease the exposition, we write in the following proposition the statement obtained when, in 
Theorem \ref{thm:main}, we suppose that $\frac{d}{R}\leq \frac{d-1}2$ and $R$ even.

\begin{proposition}\label{prop:step1}
	Fix $d\geq 2$, let $B=(B_x)_{x\in\R^d}$ be a real-valued continuous centered Gaussian field  on $\R^d$, and
	assume that $B$ is stationary, isotropic and has unit-variance.
	Let $\varphi:\R\to\R$ be such that $\E[\varphi(N)^2]<\infty$ with $N\sim N(0,1)$, let $R$ be the Hermite rank of $\varphi$
	and consider $Y_t$ defined by (\ref{Yt}), where $D$ compact and ${\rm Vol}(D)>0$. Set $m_t=\E[Y_t]$ and $\sigma_t=\sqrt{{\rm Var}(Y_t)}>0$, and 
	recall the definition  (\ref{mu}) of the isotropic spectral measure $\mu$ and the definition (\ref{var chaos bis}) of $w_{R,t}$:
	\[
	w_{R,t}=\int_{\{|z|\le t\}}C^R(z)dz.
	\]
	If (\ref{spectralcond}) holds, if  $R$ even and $\frac{d}{R}\leq \frac{d-1}2$ (i.e. $R\ge4$ and $R$ even), then $\sigma_t^2\asymp t^dw_{R,t}$ and 
	$$
	\frac{Y_t-m_t}{\sigma_t}\overset{\rm law}{\to} N(0,1)\quad \mbox{as $t\to\infty$}.
	$$
\end{proposition}

The goal of this Section \ref{step1} is to prove Proposition \ref{prop:step1}.
As we will see, it will be a direct consequence of Lemma \ref{lm3.1} and Proposition \ref{partial} below.

\medskip

We start with Lemma \ref{lm3.1}.

\begin{lemma}\label{lm3.1}
	Let $\rho$ and $\mu$ be associated as in (\ref{rho}) and consider an exponent $\beta\in (0,\frac{d-1}2]$.
	If $\int_0^\infty s^{-\beta}\mu(ds)<\infty$, then $\sup_{r\in\R_+}r^\beta|\rho(r)|<\infty$.
\end{lemma}
\noindent
{\it Proof}. Since $\rho(|x|)=\mathbb{E}[B_0B_x]$ for all $x\in\R^d$, we deduce from Cauchy-Schwarz that $|\rho|$ is bounded by 1,
and thus  $\sup_{r\in [0,T]}r^\beta|\rho(r)|<\infty$ for all fixed $T>0$.

On the other hand, using the representation (\ref{rho}), we can write
$$r^{\beta}\rho(r) = \int_0^\infty (rs)^\beta b_d(rs) \frac{\mu(ds)}{s^\beta}. $$
But $J_{\frac{d}2-1}$ is bounded and satisfies $J_{\frac{d}2-1}(u)=O(u^{-1/2})$ as $u\to\infty$ (see (\ref{boundBessel})).
So $\sup_{u\in\R_+}u^\beta |b_d(u)|<\infty$, and the desired conclusion follows.\qed

\bigskip

Now, let us state and prove Proposition \ref{partial}, which may be of independent interest.

\begin{proposition}\label{partial}
	Let $B=(B_x)_{x\in\R^d}$ be a real-valued continuous centered Gaussian field  on $\R^d$, and
	assume that $B$ is stationary and has unit-variance (note that we did not assume isotropy and $d\ge2$).
	Let $\varphi:\R\to\R$ be such that $\E[\varphi(N)^2]<\infty$ with $N\sim N(0,1)$, let $R$ be the Hermite rank of $\varphi$
	and consider $Y_t$ defined by (\ref{Yt}), where $D$ compact and ${\rm Vol}(D)>0$. Set $m_t=\E[Y_t]$ and $\sigma_t=\sqrt{{\rm Var}(Y_t)}>0$, and
	recall the definition  (\ref{C}) of the covariance $C$ and the definition (\ref{var chaos bis}) of $w_{R,t}$.
	
	If $R$ is even, if $\lim_{t\to\infty}w_{R,t}=\infty$ and if 
	\begin{equation}\label{bounded}
		\sup_{x\in\R^d} |x|^{d/R}\big|C(x)\big|<\infty,
	\end{equation}
	then $\sigma_t^2 \asymp t^dw_{R,t} $ and
	\begin{equation}\label{todopartial}
		\frac{Y_t-m_t}{\sigma_t}\overset{\rm law}{\to} N(0,1)\quad \mbox{as $t\to\infty$}.
	\end{equation}
\end{proposition}
\noindent
{\it Proof}. The proof is divided into several steps.
Recall the definition (\ref{yqt}) of $Y_{q,t}$.

\bigskip

{\it Step 1}. We claim that if $R$ even, then
\begin{equation}\label{**}
	v_{R,t}\asymp w_{R,t},
\end{equation} 
where $v_{R,t}$ is defined by (\ref{var chaos}).
To prove it, below we let $c>0$ denote a constant independent of $t$ whose value can change from one instance to another.
Using (\ref{var chaos}), we have on the one hand
\begin{eqnarray*}
	v_{R,t}&&=\frac{1}{R!}t^{-d}{\rm Var}(Y_{R,t}) = \int_{\{|z|\le{\rm diam}(D)t\}} C(z)^R g_D\left(\frac{z}{t}\right)dz\\&&\le {\rm Vol}(D)w_{R,{\rm diam}(D)t}\le c \hspace{0.5mm} w_{R,t}
\end{eqnarray*}
where the last equality follows from the so-called doubling conditions at the origin for non-negative positive definite functions (see \cite{Gorbachev}).

On the other hand, $g_D$ is uniformly continuous (in particular continuous in $0$) according to \cite{Galerne}. 
We deduce that $g_D\left(\frac{z}{t}\right)\geq {\rm Vol}(D)-\frac{1}{2}{\rm Vol}(D)=\frac{1}{2}{\rm Vol}(D)$ for all $z\in\{|z|\le \delta_{D}t\}$ for some $\delta_D>0$ depending only on $D$ and for every  $t>0$.
As a result,
\begin{eqnarray*}
	v_{R,t}=\frac{1}{R!}t^{-d} {\rm Var}(Y_{R,t}) &\geq& \int_{\{|z|\le\delta_D t\}} C(z)^R g_D\left(\frac{z}{t}\right)dz\\
	&\geq & \frac12 {\rm Vol}(D) w_{R,\delta_D t} \geq cw_{R,t},
\end{eqnarray*}
where the last inequality follows again from the doubling conditions proved in \cite{Gorbachev}. The announced claim (\ref{**}) follows.
\bigskip

{\it Step 2}. Since $w_{R,t}\to\infty$, it follows from Step 1 that $t^{-d}{\rm Var}(Y_{R,t})\to\infty$.
Moreover, since (\ref{bounded}) holds, we have that $\int_{\R^d}|C(x)|^{R+1}dx<\infty$.
Applying Proposition \ref{Prop5.4}, we obtain that $\sigma_t^2 \asymp t^dw_{R,t} $ and 
also that (\ref{todopartial}) will follow if we prove that $Y_{R,t}/\sqrt{{\rm Var}(Y_{R,t})}\to N(0,1)$.

\bigskip

{\it Step 3}. In this last step, we prove that $Y_{R,t}/\sqrt{{\rm Var}(Y_{R,t})}\to N(0,1)$, which will complete the proof of Proposition \ref{partial},
see the conclusion of Step 2. 
To do this, we use Theorem \ref{fourth moment theorem}, following the same approach as in \cite[Lemma 8.1]{NPR}. 
Fix a contraction index $r\in\{1,\dots,R-1\}$.
Using the inequality $u^rv^{R-r}\le u^{R}+v^{R}$ for $u,v\in\mathbb{R}_+$, we have
\begin{eqnarray*}
	&&\int_{\{|u|\le {\rm diam}(D)t\}^3}|C(x)|^r|C(y)|^r|C(z)|^{R-r}|C(x+y+z)|^{R-r}dxdydz \\
	&\leq&2\int_{\{|u|\le {\rm diam}(D)t\}^3}|C(x)|^r|C(y)|^R|C(x+y+z)|^{R-r}dxdydz\\
	&\leq&c \widetilde{w}_{r,t}w_{R,t}\widetilde{w}_{R-r,t}
\end{eqnarray*}
where the last inequality follows from the change of variable $a=x+y+z$ and doubling conditions in \cite{Gorbachev}, and 
\[
\widetilde{w}_{q,t}=\int_{\{|z|\le t\}}|C(z)|^qdz.
\]
We deduce that (\ref{contractions}) is bounded by
\begin{eqnarray*}
	\frac{c}{t^d}\,\frac{\widetilde{w}_{r,t}w_{R,t}\widetilde{w}_{R-r,t}}{v_{R,t}^2} = O\left(t^{-d}\frac{\widetilde{w}_{r,t}\widetilde{w}_{R-r,t}}{w_{R,t}}\right) ,
\end{eqnarray*}
where the big O comes from (\ref{**}).
We deduce from (\ref{bounded}) that $\widetilde{w}_{q,t}= O\left(t^{d-d\frac{q}{R}}\right)$ for $q<R$. Using this last fact, we get
$$
O\left(t^{-d}\frac{\widetilde{w}_{r,t}\widetilde{w}_{R-r,t}}{w_{R,t}}\right) = O\left(t^{-d}\frac{t^{d-d\frac{r}{R}}t^{d-d\frac{R-r}{R}}}{w_{R,t}}\right)= O\left(\frac{1}{w_{R,t}}\right)
$$
and then (\ref{contractions}) converges to $0$ as $t\to\infty$. Therefore, the convergence $Y_{R,t}/\sqrt{{\rm Var}(Y_{R,t})}\to N(0,1)$ follows from Theorem \ref{fourth moment theorem}, and completes the proof of Proposition \ref{partial}.\qed

\bigskip

We can now proceed with the proof of Proposition \ref{prop:step1}, that is, 
the proof of Theorem \ref{thm:main} when $R\geq 4$ and $R$ even (or equivalently, $\frac{d}{R}\leq \frac{d-1}2$ and $R$ even).

\bigskip
\noindent
{\it Proof of Proposition \ref{prop:step1}}.
If $w_{R,t}$ is convergent, then the result follows applying the Breuer-Major theorem \ref{BM}. So we can assume that $w_{R,t}\rightarrow\infty$. Now, by comparing the statements of Proposition  \ref{prop:step1} and Proposition \ref{partial}, we see that we are left to check that, if (\ref{spectralcond}) holds with $\frac dR\le \frac{d-1}{2}$, then (\ref{bounded}) holds.
Since this is a mere application of Lemma \ref{lm3.1} with $\beta= \frac{d}{R}$, the proof of Proposition \ref{prop:step1} is complete.\qed

\section{Proof of Theorem \ref{thm:main} when $R=2$}\label{step2}

This section is devoted to the proof of Theorem \ref{thm:main} when $R=2$.
It requires the introduction of novel ideas with respect to the existing literature, mainly Fourier arguments. 

To ease the exposition, we write in the following proposition the statement obtained when, in 
Theorem \ref{thm:main}, we additionally suppose that $R=2$ and $\left|\mathcal{F}[\mathbf{1}_{D}](x)\right|=O\left(\frac{1}{|x|^{d/2}}\right)$.

\begin{proposition}\label{prop:step2}
	Fix $d\geq 2$, let $B=(B_x)_{x\in\R^d}$ be a real-valued continuous centered Gaussian field  on $\R^d$, and
	assume that $B$ is stationary, isotropic and has unit-variance.
	Let $\varphi:\R\to\R$ be such that $\E[\varphi(N)^2]<\infty$ with $N\sim N(0,1)$, and assume that $\varphi$ has Hermite rank $R=2$.
	Consider $Y_t$ defined by (\ref{Yt}), where $D$ compact and ${\rm Vol}(D)>0$. Set $m_t=\E[Y_t]$ and $\sigma_t=\sqrt{{\rm Var}(Y_t)}>0$, and 
	recall the definition  (\ref{mu}) of the isotropic spectral measure $\mu$ and the definition (\ref{var chaos bis}) of $w_{q,t}$.
	
	If $|\mathcal{F}[\mathbf{1}_{D}](x)|=O\left(\frac{1}{|x|^{d/2}}\right)$ as $|x|\rightarrow\infty$\footnote{Since we assumed $D$ to be compact, this happens for example when $D=\overline{\mathring{D}}$ and $\partial D$ is smooth with non-vanishing Gaussian curvature, see e.g. \cite{Brandolini}.}, with $\mathcal{F}$ the Fourier transform, 
	and if the spectral condition holds
	\begin{equation}\label{spectralcondR=2}
		\int_0^\infty s^{-\frac{d}{2}}\mu(ds)<\infty 
	\end{equation}
	then $\sigma_t^2 \asymp t^dw_{2,t}$ and
	$$
	\frac{Y_t-m_t}{\sigma_t}\overset{\rm law}{\to} N(0,1)\quad \mbox{as $t\to\infty$}.
	$$
\end{proposition}
Before proving Proposition \ref{prop:step2} (which is the goal of this section), let us recall the definition (\ref{yqt}) of $Y_{2,t}$ and let us state and prove some preliminary results.
We start with Lemma \ref{suitexp}, reformulating in a spectral form the norm of the contractions introduced in section \ref{fmt}. Note that the following result has an analogous version for a general $r$th contraction in the case $Y_{q,t}$, but we skip this unnecessary extension for the sake of brevity.
\begin{lemma}\label{suitexp}
	Let $B=(B_x)_{x\in\R^d}$ be a real-valued continuous centered Gaussian field  on $\R^d$, and
	assume that $B$ is stationary and has unit-variance (note that we did not assume isotropy and $d\ge2$).
	Assume that $D\subset \R^d$ is compact with ${\rm Vol}(D)>0$, recall the notions introduced in section \ref{fmt} and write $$Y_{2,t}=\int_{tD}H_2(B_x)dx=I_2\left(f_{t}\right),$$  
	where $$f_{t}=\int_{tD}e_x^{\otimes 2}dx.$$Recall the definition (\ref{C}) of $C$ and define 
	\begin{equation}\label{Ct}
		C_t(u):=C(u)\mathbf{1}_{\{|u|\le {\rm diam}(D)t\}}(u), \quad t>0, \quad u\in\R^d,
	\end{equation}
and for $D\subset \R^d$ compact
\begin{equation}\label{tDu}
	D_t(u):=tD\cap (tD+u), \quad u\in\R^d.
\end{equation}
Finally, recall the definition (\ref{fourier}) of $G$. Then
\begin{eqnarray}
	\norm{f_{t}\otimes_1 f_{t}}_2^2&=&\int_{\mathbb{R}^d}G(dx)
	\int_{\mathbb{R}^d}dy\mathcal{F}[C_t](x-y)\left|\int_{\R^d} C_t(w) e^{i\langle x, w\rangle}, \mathcal{F}\left[\mathbf{1}_{D_t(-w)}\right](y) dw\right|^2 \notag\\
\label{30}
\end{eqnarray}
where $\mathcal{F}$ is the Fourier transform, and 
\begin{align}
	\int_{\R^d} C_t(w) e^{i\langle x, w\rangle} \mathcal{F}\left[\mathbf{1}_{D_t(-w)}\right](y) dw
	=\int_{\mathbb{R}^d} \mathcal{F}[C_t](x-z) \mathcal{F}[\mathbf{1}_{tD}](y-z)\mathcal{F}[\mathbf{1}_{tD}](z)dz.
	\label{31}
\end{align}
\end{lemma}
\noindent
{\it Proof}.
	Proceeding exactly as in the proof of Theorem \ref{fourth moment theorem}, we have
	\begin{eqnarray*}
		\norm{f_{t}\otimes_1 f_{t}}_{2}^2
		&=&\int_{(tD)^4}C(x_1-x_3)C(x_2-x_4)C(x_1-x_2)\\
		&&\hskip4cm \times C(x_3-x_4)dx_1dx_2dx_3 dx_4.
	\end{eqnarray*}
	Applying the change of variable $u=x_1-x_3$, $v=x_3-x_4$, $w=x_4-x_2$, $z=x_2$ we have: $x_2=z$, $x_4=w+x_2=w+z$, $x_3=v+x_4=v+w+z$ and $x_1=u+x_3=u+v+w+z$. Then 
	\begin{eqnarray*}
		&&\norm{f_{t}\otimes_1 f_{t}}_{2}^2\\
		&=&\int_{\mathbb{R}^{3d}} C(u)C(v) C(w)C(u+v+w)\\&&\times \left(\int_{\mathbb{R}^d}\mathbf{1}_{tD}(z)\mathbf{1}_{tD}(w+z)\mathbf{1}_{tD}(v+w+z)\mathbf{1}_{tD}(u+v+w+z)dz\right)du dv dw\\
		&=&\int_{\mathbb{R}^{3d}} C(u)C(v) C(w)C(u+v+w)\\&&\times \left(\int_{\mathbb{R}^d}\mathbf{1}_{tD}(z)\mathbf{1}_{tD}(-w+z)\mathbf{1}_{tD}(-v-w+z)\mathbf{1}_{tD}(-u-v-w+z)dz\right)du dv dw\\
		&=&\int_{\mathbb{R}^{3d}} C(u)C(v) C(w)C(u+v+w)\\&&\times \left(\int_{\mathbb{R}^d}\mathbf{1}_{tD\cap (tD+w)\cap (tD+w+v)\cap(tD+u+v+w)}(z)dz\right)du dv dw\\
		&=&\int_{\R^{3d}} C_t(u)C_t(v) C_t(w)C(u+v+w)\\&&\times {\rm Vol}((tD-w)\cap tD \cap (tD+v)\cap(tD+u+v))du dv dw,
	\end{eqnarray*}
	where in the last equality we used the translation invariance of the Lebesgue measure (subtracting $w$) and the definition (\ref{Ct}), justified by the fact that $tD\cap (tD+a)$ is empty when $|a|>{\rm diam}(D)t$.
	Now recall the definition (\ref{tDu}).
	We have
	\begin{eqnarray*}
		&&\norm{f_{t}\otimes_1 f_{t}}_{2}^2\\
		&=& \int_{\R^{3d}} C_t(u)C_t(v) C_t(w)C(u+v+w){\rm Vol}\left(D_t(-w) \cap (D_t(u)+v)\right)du dv dw\\
		&=& \int_{\R^{3d}}C_t(u)C_t(v) C_t(w)C(u+v+w)\left(\mathbf{1}_{D_t(-w)}\ast\mathbf{1}_{-D_t(u)}\right)(v)du dv dw ,
	\end{eqnarray*}
	where in the last expression we used that 
	\begin{eqnarray*}
		&&{\rm Vol}\left(D_t(-w) \cap (D_t(u)+v)\right)\\
		&=&\int_{\mathbb{R}^d}\mathbf{1}_{D_t(-w)}(z) \mathbf{1}_{D_t(u)}(z-v)dz=\left(\mathbf{1}_{D_t(-w)}\ast\mathbf{1}_{-D_t(u)}\right)(v).
	\end{eqnarray*}
	Now, using the spectral representation (\ref{fourier}) \[
	C(u+v+w)=\int_{\mathbb{R}^d}e^{i\langle x, u+v+w\rangle}G(dx),
	\]
	we have that 
	\begin{align}
		&\norm{f_{t}\otimes_1 f_{t}}_{2}^2\label{28}\\
		=&\int_{\mathbb{R}^d}G(dx)\int_{\R^{2d}}du dw C_t(u)C_t(w) e^{i\langle x, u+w\rangle} \int_{\R^d}e^{i\langle x, v\rangle}C_t(v)\notag\\
		&\times\left(\mathbf{1}_{D_t(-w)}\ast\mathbf{1}_{-D_t(u)}\right)(v)dv.\notag
	\end{align}
	Fix $x$, $u$ and $w$, and let us focus in (\ref{28}) on the integral with respect to $v$.
	Using the properties of the Fourier transform $\mathcal{F}$ with respect to convolution and products, we can write
	\begin{eqnarray*}
		&&\int_{\R^d}e^{i\langle x, v\rangle}C_t(v)(\mathbf{1}_{D_t(-w)}\ast\mathbf{1}_{-D_t(u)})(v)dv=\mathcal{F}[C_t\cdot(\mathbf{1}_{D_t(-w)}\ast\mathbf{1}_{-D_t(u)})](x)\\
		&=&\left(\mathcal{F}[C_t]\ast\mathcal{F}[\mathbf{1}_{D_t(-w)}\ast\mathbf{1}_{-D_t(u)}]\right)(x)=\left(\mathcal{F}[C_t]\ast\left(\mathcal{F}[\mathbf{1}_{D_t(-w)}]\mathcal{F}[\mathbf{1}_{-D_t(u)}]\right)\right)(x)\\
		&=&\int_{\mathbb{R}^d}\mathcal{F}[C_t](x-y)\mathcal{F}[\mathbf{1}_{D_t(-w)}](y)\mathcal{F}[\mathbf{1}_{-D_t(u)}](y)dy.
	\end{eqnarray*}
	Putting everything in (\ref{28}), we get
	\begin{align}
		\norm{f_{t}\otimes_1 f_{t}}_{2}^2\label{29}
		=&\int_{\mathbb{R}^d}G(dx) \int_{\R^{2d}}du dw C_t(u)C_t(w) e^{i\langle x, u+w\rangle} \\
		&\times
		\int_{\mathbb{R}^d}dy\mathcal{F}[C_t](x-y)\mathcal{F}[\mathbf{1}_{D_t(-w)}](y)\mathcal{F}[\mathbf{1}_{-D_t(u)}](y).\notag
	\end{align}
	Exchanging integrals in (\ref{29}) yields
	\begin{align*}
		&\norm{f_{t}\otimes_1 f_{t}}_{2}^2\notag\\
		=&\int_{\mathbb{R}^d}G(dx)
		\int_{\mathbb{R}^d}dy\mathcal{F}[C_t](x-y) \notag\\
		&\times\left(\int_{\R^d}dw C_t(w) e^{i\langle x, w\rangle} \mathcal{F}[\mathbf{1}_{D_t(-w)}](y) \right)\left(\int_{\R^d}du C_t(u) e^{i\langle x, u\rangle} \mathcal{F}[\mathbf{1}_{-D_t(u)}](y)\right)\notag\\
		=&\int_{\mathbb{R}^d}G(dx)
		\int_{\mathbb{R}^d}dy\mathcal{F}[C_t](x-y)\left|\int_{\R^d} C_t(w) e^{i\langle x, w\rangle} \mathcal{F}[\mathbf{1}_{D_t(-w)}](y) dw\right|^2,\notag\\
	\end{align*}
	which is exactly (\ref{30}).
	Now, let us focus on the term
	$
	\int_{\R^d} C_t(w) e^{i\langle x, w\rangle} \mathcal{F}[\mathbf{1}_{D_t(-w)}](y) dw$ in (\ref{30}).
	Using again basic  Fourier analysis and Fubini theorem, we obtain
	\begin{align*}
		&\int_{\R^d} C_t(w) e^{i\langle x, w\rangle} \mathcal{F}[\mathbf{1}_{D_t(-w)}](y) dw= \int_{\R^d} C_t(w) e^{i\langle x, w\rangle} \mathcal{F}[\mathbf{1}_{tD}\mathbf{1}_{tD-w}](y)dw\notag\\
		=&\int_{\R^d} C_t(w) e^{i\langle x, w\rangle} \left( \int_{\mathbb{R}^d} \mathcal{F}[\mathbf{1}_{tD}](y-z)\mathcal{F}[\mathbf{1}_{tD-w}](z)dz\right)dw\notag\\
		=&\int_{\R^d} C_t(w) e^{i\langle x, w\rangle} \left( \int_{\mathbb{R}^d} \mathcal{F}[\mathbf{1}_{tD}](y-z)\mathcal{F}[\mathbf{1}_{tD}](z)e^{i\langle -w,z\rangle}dz\right)dw\notag\\
		=&\int_{\mathbb{R}^d} \mathcal{F}[\mathbf{1}_{tD}](y-z)\mathcal{F}[\mathbf{1}_{tD}](z)\left(\int_{\R^d} C_t(w) e^{i\langle x-z, w\rangle}dw\right)dz\notag\\
		=&\int_{\mathbb{R}^d} \mathcal{F}[C_t](x-z) \mathcal{F}[\mathbf{1}_{tD}](y-z)\mathcal{F}[\mathbf{1}_{tD}](z)dz,
	\end{align*}
which is exactly (\ref{31}). \qed\\

\begin{lemma}\label{boundI}
	Fix $d\ge2$, let $B=(B_x)_{x\in\R^d}$ be a real-valued continuous centered Gaussian field  on $\R^d$, and
	assume that $B$ is isotropic, stationary and has unit-variance.
	Assume that $D\subset\R^d$ is compact with ${\rm Vol}(D)>0$. Recall the definition (\ref{C}) of $C$, the definition (\ref{fourier}) of $G$, the definition (\ref{var chaos bis}) of $w_{2,t}$, the definition (\ref{Ct}) of $C_t$ and the definition (\ref{tDu}) of $D_t(u)$.
	If the spectral condition (\ref{spectralcondR=2}) holds, then  
	\begin{eqnarray*}
		&&\int_{\mathbb{R}^d}G(dx)
		\int_{|x-y|\ge|x|/2}dy|\mathcal{F}[C_t](x-y)|\left|\int_{\R^d} C_t(w) e^{i\langle x, w\rangle} \mathcal{F}[\mathbf{1}_{D_t(-w)}](y) dw\right|^2\\
		&&\le {\rm const}\, w_{2,t}^{3/2}\,  t^{2d}.
	\end{eqnarray*}
\end{lemma}

\noindent
{\it Proof}. 
First, call
\begin{equation*}
	I:=\int_{\mathbb{R}^d}G(dx)
	\int_{|x-y|\ge|x|/2}dy|\mathcal{F}[C_t](x-y)|\left|\int_{\R^d} C_t(w) e^{i\langle x, w\rangle} \mathcal{F}[\mathbf{1}_{D_t(-w)}](y) dw\right|^2.
\end{equation*}
Using polar coordinates, we can write 
$$
\mathcal{F}[C_t](x)=\int_{\{|u|\le {\rm diam}(D)t\}}C(u)e^{i\langle x,u \rangle}du=t^d\int_0^{{\rm diam}(D)}dr\rho(rt)r^{d-1}\int_{S^{d-1}} e^{irt\langle x,\theta \rangle}d\theta.
$$
Thanks to (\ref{bd})-(\ref{bd2}) and (\ref{boundBessel}) we have for $|x|>1$
\[
\left|\int_{S^{d-1}} e^{i\langle x,\theta \rangle}d\theta\right| ={\rm const} \,|x|^{1-\frac{d}2}\left|J_{\frac{d}{2}-1}(|x|)\right| \leq  {\rm const}\,|x|^{1-\frac{d}2}|x|^{-\frac12} \leq \frac{{\rm const}}{\sqrt{|x|}}.
\]
We deduce
\begin{eqnarray}
	\nonumber&&|\mathcal{F}[C_t](x)|\le \frac{{\rm const}}{\sqrt{|x|}}t^d\int_0^{{\rm diam}(D)}|\rho(rt)|\frac{r^{d-1}}{\sqrt{rt}}dr=\frac{{\rm const}}{\sqrt{|x|}} \int_0^{{\rm diam}(D)t}|\rho(r)|r^{\frac{d-1}{2}+\frac{d}{2}-1}dr \\
	&&\nonumber\le \frac{{\rm const}}{\sqrt{|x|}}  t^{\frac{d}{2}-1}\int_0^{{\rm diam}(D)t}|\rho(r)|r^{\frac{d-1}{2}}dr\le \frac{{\rm const}}{\sqrt{|x|}}  t^{\frac{d-1}{2}}\sqrt{\int_0^{{\rm diam}(D)t}\rho^2(r)r^{d-1}dr}\\
	&&\nonumber=\frac{{\rm const}}{\sqrt{|x|}}t^{\frac{d-1}{2}} w^{1/2}_{R,{\rm diam}(D)t}\leq\frac{{\rm const}}{\sqrt{|x|}}t^{\frac{d-1}{2}} w^{1/2}_{R,t}
\end{eqnarray}
where the last inequality follows from the doubling conditions in \cite{Gorbachev}.
With this estimate, we have that $I$ satisfies:
\begin{eqnarray*}
	\nonumber I&\leq &\int_{\mathbb{R}^d}G(dx)
	\int_{|y-x|\ge|x|/2}dy\left|\mathcal{F}[C_t](x-y)\right|\\
	&&\times \left|\int_{\R^d}dw C_t(w) e^{i\langle x, w\rangle} \mathcal{F}[\mathbf{1}_{D_t(-w)}](y) \right|^2 \notag\\
	\nonumber&\leq&{\rm const}\cdot t^{\frac{d-1}{2}}w^{1/2}_{R,t}\int_{\mathbb{R}^d}\frac{G(dx)}{\sqrt{|x|}}	\int_{\mathbb{R}^d}dy\left|\int_{\R^d}dw C_t(w) e^{i\langle x, w\rangle} \mathcal{F}[\mathbf{1}_{D_t(-w)}](y) \right|^2 \notag\\
	&=&\nonumber {\rm const}\cdot t^{\frac{d-1}{2}}w^{1/2}_{R,t} \nonumber\int_{\R^{2d}}du dw C_t(u)C_t(w) \notag \\
	&&\times \underbrace{\left(\int_{\mathbb{R}^d}\frac{G(dx)}{\sqrt{|x|}}e^{i\langle x, u+w\rangle}\right)}_{:=\bar{C}(u+w)}	\int_{\mathbb{R}^d}dy\mathcal{F}[\mathbf{1}_{D_t(-w)}](y)\mathcal{F}[\mathbf{1}_{-D_t(u)}](y)\notag\\
	&\leq&\nonumber  {\rm const}\cdot t^{\frac{d-1}{2}}w^{1/2}_{R,t}\int_{\R^{2d}}du dw |C_t(u)C_t(w) \bar{C}(u+w)|
	\\ 
	\nonumber &&\times 	\sqrt{{\rm Vol}(tD\cap (tD+w))}\sqrt{{\rm Vol}(tD\cap (tD+u))}\\
	&\leq&\nonumber  {\rm const}\cdot t^{\frac{3}2d-\frac{1}{2}}w^{1/2}_{R,t}\int_{\R^{2d}}du dw |C_t(u)C_t(w) \bar{C}(u+w)|,
\end{eqnarray*} 
where in the third inequality we used Cauchy-Schwarz inequality and Plancherel theorem. 

Recall now that $\int_{\mathbb{R}^d}|x|^{-\frac{d}{2}}G(dx)=\int_0^\infty s^{-\frac{d}2}\mu(ds)<\infty$ by assumption, see (\ref{spectralcondR=2}). 
By definition, the spectral measure of $\bar{C}$ is $\frac{G(d\lambda)}{\sqrt{|\lambda|}}$, see (\ref{fourier}), and the associated isotropic spectral measure
is $\bar{\mu}(ds)=\frac{\mu(ds)}{\sqrt{s}}$. Since $\int_0^\infty s^{-\frac{d}2}\mu(ds)=\int_0^\infty s^{-\frac{d-1}{2}}\bar{\mu}(ds)<\infty$, we deduce from Lemma \ref{lm3.1} that
$\sup_{r\in\R_+}r^{\frac{d-1}{2}}|\bar{\rho}(r)| = \sup_{u\in\R^d}|u|^{\frac{d-1}{2}}|\bar{C}(u)|<\infty$, that is,
$\bar{C}(u)\le {\rm const}\,|u|^{-\frac{d-1}{2}}$. Using the inequality $|ab|\le a^2+b^2$, this yields
\begin{eqnarray*}
	I&\leq&\nonumber {\rm const} \cdot t^{\frac{3}{2}d-\frac{1}{2}}w^{1/2}_{R,t}\int_{\{|u|,|w|\le {\rm diam}(D)t\}}du dw C^2(u)|\bar{C}(u+w)| \\
	&\leq&\nonumber  {\rm const} \cdot t^{\frac{3}{2}d-\frac{1}{2}}w^{1/2}_{R,t}w_{R,{\rm diam}(D)t}\int_{\{|z|\le 2diam(D) t\}}  |\bar{C}(z)|	dz\\
	&\leq& \nonumber {\rm const} \cdot t^{\frac{3}{2}d-\frac{1}{2}}w^{3/2}_{R,t}\int_0^{2diam(D) t}r^{\frac{d-1}{2}} dr=  {\rm const} \cdot t^{2d}w^{3/2}_{R,t}.
\end{eqnarray*}
where in the last inequality we used again doubling conditions in \cite{Gorbachev}.
\qed \\

\begin{lemma}\label{boundII}
	Let $B=(B_x)_{x\in\R^d}$ be a real-valued continuous centered Gaussian field  on $\R^d$, and
	assume that $B$ is stationary and has unit-variance (note that we did not assume isotropy and $d\ge2$).
	Assume that $D\subset \R^d$ is compact with ${\rm Vol}(D)>0$. Recall the definition (\ref{C}) of $C$, the definition (\ref{fourier}) of $G$, the definition (\ref{var chaos bis}) of $w_{2,t}$, the definition (\ref{Ct}) of $C_t$ and the definition (\ref{tDu}) of $D_t(u)$.
	If $|\mathcal{F}[\mathbf{1}_{D}](x)|=O\left(\frac{1}{|x|^{d/2}}\right)$ as $|x|\rightarrow\infty$\footnote{Since we assumed $D$ to be compact, this happens for example when $D=\overline{\mathring{D}}$ and $\partial D$ is smooth with non-vanishing Gaussian curvature, see e.g. \cite{Brandolini}.}, with $\mathcal{F}$ the Fourier transform, 
	and if the analogous of the spectral condition (\ref{spectralcondR=2}) holds (note that here the isotropic spectral measure $\mu$ is not defined, because $B$ is not necessarily isotropic, but $G$ is defined)
	\begin{equation*}
		\int_{\R^d} |x|^{-\frac{d}{2}}G(dx)<\infty,
	\end{equation*}
	then  
	\begin{eqnarray*}
		&&\int_{\mathbb{R}^d}G(dx)
		\int_{|y|\ge|x|/2}dy|\mathcal{F}[C_t](x-y)|\left|\int_{\R^d} C_t(w) e^{i\langle x, w\rangle} \mathcal{F}[\mathbf{1}_{D_t(-w)}](y) dw\right|^2\\
		&&\le {\rm const}\, w_{2,t}^{3/2}\,  t^{2d}.
	\end{eqnarray*}
\end{lemma}

\noindent
{\it Proof}. First, by (\ref{31}) in Lemma \ref{suitexp}, we can write 
\begin{eqnarray*}
	II:&=&\int_{\mathbb{R}^d}G(dx)
	\int_{|y|\ge|x|/2}dy|\mathcal{F}[C_t](x-y)|\left|\int_{\R^d} C_t(w) e^{i\langle x, w\rangle} \mathcal{F}[\mathbf{1}_{D_t(-w)}](y) dw\right|^2\\
	&=&\int_{\mathbb{R}^d}G(dx)
	\\
	&&\int_{|y|\ge|x|/2}dy|\mathcal{F}[C_t](x-y)|\left|\int_{\mathbb{R}^d} \mathcal{F}[C_t](x-z) \mathcal{F}[\mathbf{1}_{tD}](y-z)\mathcal{F}[\mathbf{1}_{tD}](z)dz\right|^2
\end{eqnarray*}
Using Cauchy-Schwarz inequality two times (first with respect to $y$, then with respect to $z$) and then Plancherel theorem, 
$II$ satisfies, with $\norm{\cdot}_2$ the $L^2$-norm:
\begin{eqnarray*}
	II&\leq&\norm{\mathcal{F}[C_t]}_{2}\int_{\mathbb{R}^d}G(dx)\\
	&&\times \left(\int_{|y|\ge|x|/2}dy \left|\int_{\mathbb{R}^d}dz \mathcal{F}[C_t](x-z) \mathcal{F}[\mathbf{1}_{tD}](y-z)\mathcal{F}[\mathbf{1}_{tD}](z) 
	\right|^4\right)^{1/2} \\
\end{eqnarray*} 
\begin{eqnarray*}
	&\leq&\norm{\mathcal{F}[C_t]}_{2}\int_{\mathbb{R}^d}G(dx)\\
	&&\times \left(\int_{|y|\ge|x|/2}dy \norm{\mathcal{F}[C_t]}^4_{2} \left|\int_{\mathbb{R}^d}dz \left|\mathcal{F}[\mathbf{1}_{tD}]\right|^2(y-z)\left|
	\mathcal{F}[\mathbf{1}_{tD}]\right|^2(z) \right|^2\right)^{1/2}\\
	&=&\norm{\mathcal{F}[C_t]}^3_{2}\int_{\mathbb{R}^d}G(dx)\left(\int_{|y|\ge|x|/2}dy\left| \left|\mathcal{F}[\mathbf{1}_{tD}]\right|^2\ast\left|\mathcal{F}[\mathbf{1}_{tD}]\right|^2(y)\right|^2 \right)^{1/2}\\
	&\leq&\norm{\mathcal{F}[C_t]}^3_{2}\int_{\mathbb{R}^d}G(dx)\left(\sup_{|y|\ge|x|/2}\left|\mathcal{F}[\mathbf{1}_{tD}]\right|^2\ast\left|\mathcal{F}[\mathbf{1}_{tD}]
	\right|^2(y)\right)^{1/2}\\
	&&\times \left(\int_{\mathbb{R}^d}dy\left| \left|\mathcal{F}[\mathbf{1}_{tD}]\right|^2\ast\left|\mathcal{F}[\mathbf{1}_{tD}]\right|^2(y)\right| \right)^{1/2} \\
	&\leq& \norm{\mathcal{F}[C_t]}^3_{2}\int_{\mathbb{R}^d}G(dx)\left(\sup_{|y|\ge|x|/2}\left|\mathcal{F}[\mathbf{1}_{tD}]\right|^2\ast\left|\mathcal{F}[\mathbf{1}_{tD}]
	\right|^2(y)\right)^{1/2}\\
	&&\times \left(\int_{\mathbb{R}^d}\left| \mathcal{F}[\mathbf{1}_{tD}](y)\right|^2dy\right)\\
	&=&  {\rm const} \cdot t^{d}\norm{\mathcal{F}[C_t]}^3_{2}\int_{\mathbb{R}^d}G(dx)\left(\sup_{|y|\ge|x|/2}\left|\mathcal{F}[\mathbf{1}_{tD}]\right|^2\ast\left|\mathcal{F}[\mathbf{1}_{tD}]\right|^2(y)\right)^{1/2},
\end{eqnarray*}
where the last inequality comes from Young convolution inequalities and the last equality from Plancherel theorem.

Since by assumption $|\mathcal{F}[\mathbf{1}_{D}](x)|=O\left(\frac{1}{|x|^{d/2}}\right)$ as $|x|\rightarrow\infty$, in particular $\sup_{x\in\R^d}|x|^{d/2}|\mathcal{F}[\mathbf{1}_{D}](x)|<\infty$.
We deduce, for all $t>0$ and all $y\in\R^d$, that
$$
|\mathcal{F}[\mathbf{1}_{tD}](y)|^2 =t^{2d}|\mathcal{F}[\mathbf{1}_{D}](ty)|^2 \leq {\rm const}\, t^d|y|^{-d}.
$$
This implies  
\begin{eqnarray}
	&& \nonumber \sup_{|y|\ge|x|/2}\left|\mathcal{F}[\mathbf{1}_{tD}]\right|^2\ast\left|\mathcal{F}[\mathbf{1}_{tD}]\right|^2(y)=\sup_{|y|\ge|x|/2}\left|\int_{\mathbb{R}^d}\left|\mathcal{F}[\mathbf{1}_{tD}]\right|^2(y-z)\left|\mathcal{F}[\mathbf{1}_{tD}]\right|^2(z)dz\right| \\
	&\leq& \nonumber {\rm const}\cdot \frac{t^{d}}{|x|^{d}}\sup_{|y|\ge|x|/2}\left|\int_{|z|\ge|x|/4}\left|\mathcal{F}[\mathbf{1}_{tD}]\right|^2(y-z)dz+\int_{|y-z|\ge|x|/4}\left|\mathcal{F}[\mathbf{1}_{tD}]\right|^2(z)dz\right|\\
	&\leq& \nonumber {\rm const}\cdot \frac{t^{d}}{|x|^{d}}  \int_{\mathbb{R}^d}\left| \mathcal{F}[\mathbf{1}_{tD}]\right|^2(z)dz={\rm const} \frac{t^{2d}}{|x|^{d}}.
\end{eqnarray}
We deduce that
\begin{eqnarray*}
	\nonumber II&\leq &{\rm const} \cdot t^{2d}\norm{\mathcal{F}[C_t]}^3_{2}\int_{\mathbb{R}^d}\frac{G(dx)}{|x|^{d/2}}\\
	&=&{\rm const} \cdot t^{2d}w^{3/2}_{R,{\rm diam}(D)t}\int_{\mathbb{R}^d}\frac{G(dx)}{|x|^{d/2}}\le {\rm const} \cdot t^{2d}w^{3/2}_{R,t}\int_{\mathbb{R}^d}\frac{G(dx)}{|x|^{d/2}},
\end{eqnarray*}
where the last equality comes from Plancherel theorem and the last inequality from doubling conditions in \cite{Gorbachev}. 

Since $\int_{\mathbb{R}^d}|x|^{-\frac{d}2}G(dx)<\infty$ by assumption,
our bound for $II$ is
\[
II\leq {\rm const} \cdot t^{2d}w^{3/2}_{R,t}.
\]
and the proof is concluded.
\qed \\

Now we can proceed with the proof of Proposition \ref{prop:step2}.\\

\noindent
{\it Proof of Proposition \ref{prop:step2}}. 
First, we assume without loss of generality that $w_{2,t}\rightarrow\infty$, since otherwise the statement follows from theorem \ref{BM}.
Moreover, throughout all the proof we freely use that $v_{2,t}\asymp w_{2,t}$, see (\ref{**}).

Starting from now, the proof is divided into three steps.\\

{\it Step 1: Reduction of the proof}. 
We claim that
it is enough to check that $Y_{2,t}/\sqrt{{\rm Var}(Y_{2,t})}\rightarrow N(0,1)$
in order to prove Proposition \ref{prop:step2}.
Indeed, since (\ref{spectralcondR=2}) holds and given that $\frac{d}{2}>\frac{d-1}{2}$, we deduce from Lemma \ref{lm3.1} that
$\sup_{r\in\R_+}\{r^{\frac{d-1}{2}}\rho(r)\}<\infty$. Proposition \ref{Prop5.4} implies the statement on $\sigma^2_t$  and justifies that we are left to prove that $Y_{2,t}/\sqrt{{\rm Var}(Y_{2,t})}\rightarrow N(0,1)$.

As done in the proof of Proposition \ref{prop:step1}, in order to prove the convergence $Y_{2,t}/\sqrt{{\rm Var}(Y_{2,t})}\rightarrow N(0,1)$ we make use of the Fourth Moment Theorem (\cite[Theorem 5.2.7]{bluebook});
this requires checking that the only involved contraction goes to zero. 
To do this, we will use the novel ideas from Fourier analysis introduced in Lemma \ref{suitexp}, Lemma \ref{boundI} and Lemma \ref{boundII}.

As in the proof of Theorem \ref{fourth moment theorem}, we can first 
rewrite $Y_{2,t}$ as a double Wiener-It\^o integral with respect to $B$:
$$Y_{2,t}=I_2(f_{t}),\quad
\mbox{where }
f_{t}=\int_{tD}e_x^{\otimes 2}dx,$$
with $e_x$ such that $B_x=I_1(e_x)$.
We know from the  Fourth Moment Theorem stated in section \ref{fmt} 
that $Y_{2,t}/\sqrt{{\rm Var}(Y_{2,t})}\to N(0,1)$ if and only if
$
\norm{f_{t}\otimes_1 f_{t}}_{2}/{\rm Var}(Y_{2,t})\rightarrow 0
$. 

\bigskip

{\it Step 2: A two-term error bound for the norm of the contraction}. Here we apply Lemma \ref{suitexp}.
Noting that for $x,y\in\R^d$ either $|x-y|\ge |x|/2$ or $|y|\ge |x|/2$, we deduce from (\ref{30}) and (\ref{31}) the following two terms error bound:
\begin{align*}
	&\norm{f_{t}\otimes_1 f_{t}}_{2}^2\notag\\
	\leq&\int_{\mathbb{R}^d}G(dx)
	\int_{|x-y|\ge|x|/2}dy\left|\mathcal{F}[C_t](x-y)\right| \left|\int_{\R^d} C_t(w) e^{i\langle x, w\rangle} \mathcal{F}[\mathbf{1}_{D_t(-w)}](y) dw\right|^2\notag\\
	&+\int_{\mathbb{R}^d}G(dx)
	\int_{|y|\ge|x|/2}dy\left|\mathcal{F}[C_t](x-y)\right| \left|\int_{\R^d} C_t(w) e^{i\langle x, w\rangle} \mathcal{F}[\mathbf{1}_{D_t(-w)}](y) dw\right|^2\notag\\
	=&I+II.\notag\\
	\label{32}
\end{align*}

{\it Step 3: The norm of the contraction divided by the variance goes to $0$}.
To conclude the proof of Proposition \ref{prop:step2}, we will now check (see the conclusion of Step 1) that 
$\norm{f_{t}\otimes_1 f_{t}}_{2}/{\rm Var}(Y_{2,t})\rightarrow 0$.
From (\ref{**}), we have
\begin{equation}\label{af}
	{\rm Var}(Y_{2,t})\asymp t^dw_{R,t}.
\end{equation}
We deduce from Step 2, Lemma \ref{boundI} and Lemma \ref{boundII}, that
$$
\frac{\norm{f_{t}\otimes_1 f_{t}}_{2}^2}{{\rm Var}(Y_{2,t})^2}
\leq {\rm const} \cdot \frac{w^{3/2}_{R,t}}{w^{2}_{R,t}}= {\rm const}\cdot \frac{1}{w^{1/2}_{R,t}},
$$
and the right-hand size goes to $0$, since we assumed at the beginning of the proof that $w_{2,t}\rightarrow\infty$.
\qed

\section{Proof of theorem \ref{thm:main} when $R=1$}
\label{rank1section}

This section is devoted to the proof of Theorem \ref{thm:main} in the remaining cases, namely $\varphi$ non-odd, $R=1$ and $R'\neq 3$.

To ease the exposition, we write in the following proposition the statement obtained when, in 
Theorem \ref{thm:main}, we additionally suppose that we are in the cases just mentioned above. 

\begin{proposition}\label{prop:step3}
	Fix $d\geq 2$, let $B=(B_x)_{x\in\R^d}$ be a real-valued continuous centered Gaussian field  on $\R^d$, and
	assume that $B$ is stationary, isotropic and has unit-variance.
	Let $\varphi:\R\to\R$ be not odd and such that $\E[\varphi(N)^2]<\infty$ with $N\sim N(0,1)$.
	Assume $\varphi$ has Hermite rank $R=1$ and let $R'\ge2$, $R'\neq3$, be its second Hermite rank.
	Consider $Y_t$ defined by (\ref{Yt}), where $D$ is compact and ${\rm Vol}(D)>0$. Set $m_t=\E[Y_t]$ and $\sigma_t=\sqrt{{\rm Var}(Y_t)}>0$, and 
	recall the definition  (\ref{mu}) of the isotropic spectral measure $\mu$,
	the definition (\ref{yqt}) of $Y_{q,t}$ and the definition (\ref{var chaos bis}) of $w_{q,t}$.
	Assume that the spectral condition holds
	\begin{equation}\label{spectralcond1}
		\int_0^\infty s^{-d}\mu(ds)<\infty .
	\end{equation}
	and that $|\mathcal{F}[\mathbf{1}_{D}](x)|=o\left(\frac{1}{|x|^{d/2}}\right)$ as $|x|\rightarrow\infty$ \footnote{Since we assumed $D$ to be compact, this happens for example when $D=\overline{\mathring{D}}$ and $\partial D$ is smooth with non-vanishing Gaussian curvature, see e.g. \cite{Brandolini}.}, with $\mathcal{F}$ the Fourier transform. Then \begin{equation*}
		\sigma_t^2 \asymp \left\{\begin{array}{lll}
			t^dw_{R',t}&&\mbox{if $R'\in\{2,4\}$}\\
			t^d&&\mbox{if $R'\ge5$}
		\end{array}
		\right.
	\end{equation*}  and 
	$$
	\frac{Y_t-m_t}{\sigma_t}\overset{\rm law}{\to} N(0,1)\quad \mbox{as $t\to\infty$}.
	$$
\end{proposition}

\begin{proof}
	Set $\widehat{\varphi}(x)=\varphi(x)-a_0-a_1x$. By the very definition of $R'$, the function $\widehat{\varphi}$ has Hermite rank $R'$. Let us define 
	\begin{equation}\label{Yt'}
		\widehat{Y}_t=\int_{tD}\widehat{\varphi}(B_x)dx=\sum_{q=R'}^{\infty}a_qY_{q,t}=Y_t-\E[Y_t]-a_1Y_{1,t}
	\end{equation}
	and its variance
	\begin{equation}\label{sigma'}
		{\widehat{\sigma}_t}^{2}={\rm Var}(\widehat{Y}_t)=\sum_{q=R'}^{\infty}a^2_q{\rm Var}(Y_{q,t})=\sigma_t^2-a^2_1{\rm Var}(Y_{1,t}).
	\end{equation}
	
	The proof is divided into two steps. 
	
	\bigskip
	
	{\it Step 1: CLT for $\widehat{Y}_t$}. We claim that $\widehat{\sigma}_t^2\asymp t^dw_{R',t}$ and 
	\begin{equation}
		\label{clt'}
		\frac{\widehat{Y}_t-m_t}{\widehat{\sigma}_t}\overset{\rm law}{\to} N(0,1)\quad \mbox{as $t\to\infty$}.
	\end{equation}
	First of all, observe that (\ref{spectralcond1}) implies  (\ref{spectralcond}) for $R'\geq 2$. Moreover, note that we can have three different situations:
	\begin{itemize}
		\item If $R'\ge5$, then by (\ref{spectralcond1}) and Lemma \ref{lm3.1} we have $\sup_{r\in\R_+}r^{\frac{d-1}{2}}|\rho(r)| = \sup_{u\in\R^d}|u|^{\frac{d-1}{2}}|C(u)|<\infty$, $C\in L^{R'}(\mathbb{R}^d)$, and the claim follows immediately by Theorem \ref{BM} and the fact that $\varphi$ is not odd.
		\item If $R'=4$, then the claim follows by Proposition \ref{prop:step1}.
		\item If $R'=2$, since $|\mathcal{F}[\mathbf{1}_{D}](x)|=o\left(\frac{1}{|x|^{d/2}}\right)$ as $|x|\rightarrow\infty$ implies $|\mathcal{F}[\mathbf{1}_{D}](x)|=O\left(\frac{1}{|x|^{d/2}}\right)$ as $|x|\rightarrow\infty$, then the claim follows by Proposition \ref{prop:step2}.
	\end{itemize}
	
	\bigskip
	
	{\it Step 2: CLT for $Y_t$}. We claim that $\sigma_t\sim\widehat{\sigma}_t$ as $t\to\infty$ and 
	\begin{equation}\label{red}
		\mathbb{E}\left[\left(\frac{Y_t-m_t}{\sigma_t}-\frac{\widehat{Y}_t}{\widehat{\sigma}_t}\right)^2\right]\rightarrow0.
	\end{equation}
	The proof of Proposition \ref{prop:step3} thus follows as soon as these two claims are shown to be true.
	From (\ref{extrank1}) we have 
	$$t^{-d}{\rm Var}(Y_{1,t})={\rm const}\int_{\mathbb{R}^d}\frac{G(d\lambda)}{|\lambda|^d}|t\lambda|^d|\mathcal{F}[\mathbf{1}_{D}](t\lambda)|^2. $$
	Then, combining (\ref{spectralcond1}), dominated convergence theorem and the fact that $|\mathcal{F}[\mathbf{1}_{D}](x)|=o\left(\frac{1}{|x|^{d/2}}\right)$ as $|x|\rightarrow\infty$, we have that ${\rm Var}(Y_{1,t})=o(t^d)$ as $t\rightarrow\infty$.
	Since in Step 1 we proved that $\widehat{\sigma}_t^2\asymp t^d$ or $\widehat{\sigma}_t^2\asymp t^dw_{R',t}$, we have by (\ref{sigma'}) that $\widehat{\sigma}_t\sim\sigma_t$. It remains to prove (\ref{red}). By orthogonality of chaotic projections we have
	\[
	\mathbb{E}\left[\left(\frac{Y_t-m_t}{\sigma_t}-\frac{\widehat{Y}_t}{\widehat{\sigma}_t}\right)^2\right]=a_1^2\frac{{\rm Var}(Y_{1,t})}{{\rm Var}(Y_t)}+\left(\frac{\widehat{\sigma}_t}{\sigma_t}-1\right)^2.
	\]
	Since ${\rm Var}(Y_{1,t})=o(t^d)$, the first addend converges to $0$. On the other hand, the second term vanishes because ${\widehat{\sigma}_t}\sim\sigma_t$ as $t\to\infty$.
	
	\bigskip

\end{proof}

\section{An example of application of Theorem \ref{thm:main}}\label{Examples-sec}
In this section,
we illustrate a possible use of Theorem \ref{thm:main}.
In order to introduce our class of fields of interest, recall the definition of the Bessel function $J_\nu$ given in Section \ref{besselstuff} and define, for every $\nu\ge0$, the normalized Bessel function function $\rho_\nu:[0,\infty)\rightarrow\R$ as
\[
\rho_\nu(r)=c_{\nu}\frac{J_{\nu}(r)}{r^{\nu}},
\]
where $c_\nu$ is chosen so that $\rho_\nu(0)=1$. 

Note that $\rho_{\frac d2 -1}$ is equal to the function $b_d$ defined in (\ref{bd2}), and in particular $\rho_0=b_2=J_0$ when $d=2$. 

Throughout this section, we define the \textbf{Bessel Gaussian field of order $\nu$ and dimension $d$} 
as the real-valued continuous centered Gaussian field $B^\nu=(B^\nu_x)_{x\in\R^d}$ with covariance function 
$$
\mathbb{E}[B^\nu_xB^\nu_y]=\rho_\nu(|x-y|).
$$
In particular, {\bf $\mathbf{d}$-dimensional Berry's Random Wave Model} is defined as the
Bessel Gaussian field of order $\frac d2 -1$ and dimension $d$.

The existence of the Bessel Gaussian field of order $\nu$ and dimension $d\geq 2$ is neither obvious, nor always true.
The following result provides a complete picture, 
and also shows that $d$-dimensional Berry's Random Wave Model appears to be the critical case for the existence of the Bessel Gaussian field. 

\begin{proposition}
	\label{thm:bessel}
	Fix $d\geq 2$ and $\nu\geq 0$. There exists a Bessel Gaussian field $B^\nu=(B^\nu_x)_{x\in\R^d}$ of order $\nu$ 
	and dimension $d$ if and only if $\nu\ge \frac d2 -1$. In this case, the isotropic spectral measure 
	associated to $B^\nu$ (see (\ref{rho})) is 
	\begin{equation}
		\mu_{\nu}(ds)= \left\{\begin{array}{lll}
			c_{d,\nu}s^{d-1}(1-s^2)^{\nu-\frac d2 }\mathbf {1}_{(0,1)}ds&&\mbox{if $\nu>\frac d2 -1$}\\
			\delta_1(ds)&&\mbox{if $\nu= \frac d2 -1$}
		\end{array}
		\right.,
	\end{equation}
	where $c_{d,\nu}>0$ is chosen so that $\mu_\nu$ is a probability measure.
\end{proposition}
\begin{proof} 
	Observe that $\rho_\nu$ has representation
	\[
	\rho_\nu(r)=c_{\nu}\frac{J_{\nu}(r)}{r^{\nu}}={\rm const}\times \sum_{j=0}^\infty (-1)^j \frac{\Gamma(\nu +1)(r^2/4)^j}{j!\Gamma(j+\nu+1)}.
	\]
	Combining this representation with \cite[Proposition 2.2]{Golinski}, one has that $\rho_\nu$ is not positive definite when $\nu<\frac d2 -1$, showing that the Bessel Gaussian field does not exist in this case.
	
	The statement for the critical case $\nu=\frac d2 -1$ immediately follows from (\ref{rho}) and (\ref{bd2}), after observing that $\rho_{\frac d2 -1}=b_d$. 
	
	For $\nu> \frac d2 -1$, one can actually check that $\mu_\nu$ given in the statement is the distribution of the square root of the Beta random variable $\beta(\frac d2,\nu -\frac d2+1)$. Then, using Fubini and $b_d=\rho_{\frac d2 -1}$ we have
	
	\begin{eqnarray*}
		&&\int_{0}^{1}b_d(rt)\mu_\nu(dr)={\rm const}\int_{0}^{1}\left(\sum_{k=0}^{\infty}\frac{(-1)^k\Gamma(d/2)}{\Gamma(k+d/2)k!}((rt)^2/4)^k \right)\mu_\nu(dr)\\&=&{\rm const}\sum_{k=0}^{\infty}\frac{(-1)^k\Gamma(d/2)}{\Gamma(k+d/2)k!}\left(\frac{t}{2}\right)^{2k}\int_{0}^{1}r^{2k}\mu_\nu(dr)\\&=&{\rm const}\sum_{k=0}^{\infty}\frac{\Gamma(\nu+1)(-1)^k}{\Gamma(\nu+k+1)k!}\left(\frac{t}{2}\right)^{2k}=\rho_\nu(t),
	\end{eqnarray*}
	where the second-last equality comes from the fact that the $k$-th moment of a random variable $\beta(\frac d2,\nu -\frac d2+1)$ is 
	\[
	\frac{\Gamma(d/2+k)\Gamma(\nu+1)}{\Gamma(d/2)\Gamma(\nu+k+1)}.
	\]
	This means (see (\ref{rho})) that $\rho_\nu$ defines a Bessel Gaussian field when $\nu>\frac{d}{2}-1$ and that its spectral measure is $\mu_\nu$ as defined in the statement.
	
	The fact that any Bessel Gaussian field is continuous can be proved using standard results, see e.g. \cite[Theorem 1.4.1]{Adler}.
\end{proof}
Now let us apply our Theorem \ref{thm:main} to the Bessel Gaussian field with parameter $\nu\ge \frac d2 -1$ and dimension $d\geq 2$. 

\begin{corollary}[Spectral CLT applied to $B^\nu$]
	Fix $d\geq 2$ and $\nu\ge \frac d2 -1$. Let $B^\nu=(B^\nu_x)_{x\in\R^d}$ be a Bessel Gaussian field  on $\R^d$.
	Let $\varphi:\R\to\R$ be not odd and such that $\E[\varphi(N)^2]<\infty$ with $N\sim N(0,1)$, let $R$ be the Hermite rank of $\varphi$
	and consider $Y_t$ defined by (\ref{Yt}), where $B=B^\nu$ and $D$ is compact, $Vol(D)>0$. Set $m_t=\E[Y_t]$ and $\sigma_t=\sqrt{{\rm Var}(Y_t)}>0$.
	If $R=2$, assume that $|\mathcal{F}[\mathbf{1}_{D}](x)|=O\left(\frac{1}{|x|^{d/2}}\right)$ as $|x|\rightarrow\infty$. Then, if $R=2$ or $R\ge4$, we have that $\sigma_t^2\asymp t^d w_{R,t}$ and
	$$
	\frac{Y_t-m_t}{\sigma_t}\overset{\rm law}{\to} N(0,1)\quad \mbox{as $t\to\infty$}.
	$$
\end{corollary}
\begin{proof}
	Using standard asymptotics for Bessel functions (see e.g. \cite{krasikov}), one has 
	\[
	\rho_\nu(r)=c_\nu \frac{\cos(r-a_\nu)}{r^{\nu+1/2}}+O\left(\frac{1}{r^{\nu +3/2}}\right) \quad \mbox{as $r\rightarrow\infty$}
	\]
	for some constants  $c_\nu$ and $a_\nu$. Then, it follows that 
	\[
	\int_0^{2t}|\rho_\nu(r)|^qr^{d-1}dr=O\left( \int_1^{2t}\frac{dr}{r^{q\nu+\frac q2 +1-d}}\right) \quad  \mbox{as $t\rightarrow\infty$}.
	\]
	Since $\nu\geq \frac d2 -1$, when $R\geq 5$ we have that $|\rho|^q(r)r^{d-1}$ is integrable, and  the result follows by Breuer Major Theorem \ref{BM}. For $R=2$ and $R=4$, the result follows by a direct application of Theorem \ref{thm:main}, since the negative moment of order $d/R$ of $\mu_\nu$ exists for every $d\ge2$ and $R\ge2$. 
\end{proof}

\begin{remark}
	For simplicity, in the previous statement we did not consider all the cases which can be solved applying our Theorem \ref{thm:main}. For example, the spectral condition (\ref{spectralcond}) for $R=1$ is not verified by $\mu_\nu$ if $\nu>\frac{d}{2}-1$, but holds in the $d$-dimensional Berry case $\nu=\frac{d}{2}-1$. 
Moreover, relying on a slight variation of the proof of Proposition \ref{prop:step3}, it would have not been very difficult to also cover the non-critical cases where $\nu\in[\frac d2-1,\frac{d}{2}-\frac 12)$. We skip the details for the sake of brevity.
\end{remark}

%%%%%%%%%%%%%%%%%%%%%%%%%%%%%%%%%%%%%%%%%%%%%%
%% Funding information, if any,             %%
%% should be provided in the                %%
%% funding section.                         %%
%%%%%%%%%%%%%%%%%%%%%%%%%%%%%%%%%%%%%%%%%%%%%%
\section*{Acknowledgments}\smallskip

We would like to thank the referee for a careful reading, constructive
remarks and useful suggestions.
L. Maini was supported by the Luxembourg National Research Fund PRIDE17/1224660/GPS. I. Nourdin was supported by the Luxembourg National Research O18/12582675/APOGee.

%%  MR numbers will be added by VTeX.                      %%
%%                                                         %%
%%  Use \cite{...} to cite references in text.             %%
%%                                                         %%
%%%%%%%%%%%%%%%%%%%%%%%%%%%%%%%%%%%%%%%%%%%%%%%%%%%%%%%%%%%%%

%% if your bibliography is in bibtex format, uncomment commands:
%\bibliographystyle{imsart-number} % Style BST file (imsart-number.bst or imsart-nameyear.bst)
%\bibliography{bibliography}       % Bibliography file (usually '*.bib')

\bibliographystyle{plain}

\end{document}